\documentclass{amsart}
\usepackage{amssymb, amsmath, amsthm, graphics, comment, xspace, enumerate}
\newtheorem{thm}{Theorem}[section]

\newtheorem{lem}[thm]{Lemma}
\newtheorem{prop}[thm]{Proposition}
\theoremstyle{definition}
\newtheorem{defn}[thm]{Definition}
\newtheorem{example}[thm]{Example}
\theoremstyle{remark}
\newtheorem{rem}[thm]{Remark}
\numberwithin{equation}{section}

\newcommand{\DD}{\mathcal D}
\newcommand{\SSS}{\mathcal S}

\begin{document}
\title[$C$-distribution semigroups and $C$-ultradistribution semigroups...]{$C$-distribution semigroups and $C$-ultradistribution semigroups
in locally convex spaces}
\author{Marko Kosti\' c}
\address{Faculty of Technical Sciences,
University of Novi Sad,
Trg D. Obradovi\' ca 6, 21125 Novi Sad, Serbia}
\email{marco.s@verat.net}

\author{Stevan Pilipovi\' c}
\address{Department for Mathematics and Informatics,
University of Novi Sad,
Trg D. Obradovi\' ca 4, 21000 Novi Sad, Serbia}
\email{pilipovic@dmi.uns.ac.rs}

\author{Daniel Velinov}
\address{Department for Mathematics, Faculty of Civil Engineering, Ss. Cyril and Methodius University, Skopje,
Partizanski Odredi
24, P.O. box 560, 1000 Skopje, Macedonia}
\email{velinovd@gf.ukim.edu.mk}

{\renewcommand{\thefootnote}{} \footnote{2010 {\it Mathematics
Subject Classification.} 47D06, 47D60, 47D62, 47D99, 45N05.
\\ \text{  }  \ \    {\it Key words and phrases.} $C$-distribution semigroups, $C$-ultradistribution semigroups,
integrated $C$-semigroups, convoluted $C$-semigroups, well-posedness, locally convex spaces.
\\  \text{  }  \ \ This research is partially supported by grant 174024 of Ministry
of Science and Technological Development, Republic of Serbia.}}

\begin{abstract}
The main purpose of this paper is to investigate $C$-distribution semigroups and $C$-ultradistribution semigroups in the setting of
sequentially complete locally convex spaces. There are a few important theoretical novelties in this field and there are given  some interesting examples. Stationary dense operators in sequentially complete locally convex space are considered.

\end{abstract}
\maketitle

\section{Introduction and Preliminaries}
It is well known that the class of distribution semigroups in Banach spaces was introduced by J. L. Lions \cite{li121} in 1960
as an attempt to seek for the solutions of abstract first order differential equations that are not well-posed in the usual sense, i.e.,
whose solutions are not governed by strongly continuous semigroups of linear operators. From then on,
distribution semigroups have attracted the attention of a large number of mathematicians (cf. \cite{a22}, \cite{b52}-\cite{barbu1}, \cite{fat1}-\cite{fujiwara}, \cite{lar},  \cite{peet} and \cite{yosi} for more details about distribution semigroups in Banach spaces with densely defined generators). Following the pioneering works by D. Prato, E. Sinestrari \cite{d81}, W. Arendt \cite{a11} and E. B. Davies, M. M. Pang \cite{d811qw} (cf. also S. \=Ouchi \cite{o192}), there has been growing interest in dropping the usually imposed density assumptions in the theory of first order differential equations and discussing various generalizations of strongly continuous semigroups, such as integrated semigroups, $C$-regularized semigroups and $K$-convoluted semigroups (cf. \cite{knjigah} for a comprehensive survey of results).
The class of distribution semigroups with not necessarily densely defined generators has been introduced independently by P. C. Kunstmann \cite{ku112} and S. W. Wang \cite{w241}, while the class of $C$-distribution semigroups has been introduced by the first named author in  \cite{ko98}  (cf.
\cite{ki90}, \cite{knjigah}, \cite{ku101}, \cite{ku112},
\cite{kmp} and \cite{me152}-\cite{mija} for further information in this direction).  Ultradistribution semigroups in Banach spaces, with densely or non-densely defined generators, and abstract Beurling spaces
have been analyzed in the papers of R. Beals \cite{b41}-\cite{b42}, J. Chazarain \cite{cha}, I. Cior\u anescu \cite{ci1}, I. Cior\u anescu, L.  Zsido \cite{cizi}, P. R. Chernoff \cite{chernof}, H. A. Emami-Rad
\cite{er} and H. Komatsu \cite{k92} (cf. also \cite{knjigah}, \cite{diff1}, \cite{ptica}-\cite{tica}, \cite{ku113} and \cite{me152}).
On the other hand, the study of distribution semigroups in locally convex spaces has been initiated by R. Shiraishi, Y. Hirata \cite{1964}, T. Ushijima \cite{ush1} and M. Ju Vuvunikjan \cite{vuaq}. For the best knowledge of the authors, there is no significant reference which treats ultradistribution semigroups in locally convex spaces.\\

In this paper, we introduce and systematically analyze the classes of\\ $C$-distribution semigroups and $C$-ultradistribution semigroups in the setting of
sequentially complete locally convex spaces, providing additionally a large amount of relevant references on the subjects under consideration. We provide a few theoretical novelties. For example,
the notion of a pre-distribution semigroup and the notion of a non-dense distribution semigroup seem to be completely new and not considered elsewhere, with the exception of classical case that $E$ is a Banach space.

The definition of a stationary dense operators is continuation of an investigation on dense distribution semigroups. By studying $A$ in a sequentially complete locally convex space $E$ we can study the behavior of $A_{\infty}$ in $D_{\infty}(A)$. Naturally, we can pose the question are they somehow related. The answer is affirmative. If $A$ is stationary dense for a closed linear operator with some additional condition, the information lost in passing from $A$ in $E$ to $A_{\infty}$ in $D_{\infty}(A)$ can be retrieved. In case of a dense ultradistribution semigroups the definition of stationary dense operators can not be considered.\\
In \cite{ku101} are defined stationary dense operators in Banach space $E$ and their connections to dense distribution semigroups. When we are dealing with semigroups on locally convex spaces in order to exist the resolvent of $A$, where $A$ is an infinitesimal generator for the semigroup, we suppose that the semigroups are equicontinuous. Further on, it will be given new examples on stationary dense operators on sequentially complete locally convex spaces and results of \cite{ku101}, \cite{ush}, \cite{fujiwara} will be extended.

The organization of paper can be briefly described as follows. In Section 2, we analyze the $C$-wellposedness of first order Cauchy problem in the sense of distributions and ultradistributions, paying special accent on the study of $C$-generalized resolvents of linear operators in Subsection 2.1. Section 3 is devoted to the study of main structural properties of $C$-distribution semigroups and $C$-ultradistribution semigroups. In section 4 are considered stationary dense operators in sequentially complete locally convex spaces following the investigation in \cite{ku101}.

\subsection{Notation}
We use the standard notation throughout the paper.
Unless specified otherwise,
we shall always assume
that $E$ is a Hausdorff sequentially complete
locally convex space over the field of complex numbers, SCLCS for short.
If $X$ is also a SCLCS, then we denote by
$L(E,X)$ the space consisting of all continuous linear mappings from $E$ into
$X;$ $L(E)\equiv L(E,E).$ By $\circledast_{E}$ ($\circledast$, if there is no risk for confusion), we denote the fundamental system of seminorms which defines the topology of $E.$
By $L_{\circledast}(E)$ we denote the subspace of $L(E)$ consisting of those continuous linear mappings $T$ from $E$ into $E$ satisfying that for each $p\in \circledast$ there exists $c_{p}>0$ such that
$p(Tx)\leq c_{p} p(x),$ $x\in E.$
Let ${\mathcal B}$ be the family of bounded subsets\index{bounded subset} of $E,$ and
let $p_{B}(T):=\sup_{x\in B}p(Tx),$ $p\in \circledast_{X},$ $B\in
{\mathcal B},$ $T\in L(E,X).$ Then $p_{B}(\cdot)$ is a seminorm\index{seminorm} on
$L(E,X)$ and the system $(p_{B})_{(p,B)\in \circledast_{X} \times
{\mathcal B}}$ induces the Hausdorff locally convex topology on
$L(E,X).$
If $E$ is a Banach space, then we denote by $\|x\|$ the norm of an element $x\in E.$
The Hausdorff locally convex topology on $E^{\ast},$ the dual space\index{dual space} of $E,$
defines the
system $(|\cdot|_{B})_{B\in {\mathcal B}}$ of seminorms on
$E^{\ast},$ where and in the sequel $|x^{\ast}|_{B}:=\sup_{x\in
B}|\langle x^{\ast}, x \rangle |,$ $x^{\ast} \in E^{\ast},$ $B\in
{\mathcal B}.$ Here $\langle \ , \ \rangle $ denotes the duality bracket between
$E$ and $E^{\ast},$ sometimes we shall also write $\langle x , x^{\ast} \rangle$
or $x^{\ast}(x)$ to denote the value of
$\langle x^{\ast}, x \rangle .$
Let us recall that the spaces $L(E)$ and $E^{\ast}$ are sequentially
complete provided that $E$ is barreled\index{barreled space} (\cite{meise}). By $E^{\ast \ast}$ we denote the bidual of $E.$
Recall, the polars of nonempty sets $M\subseteq E$ and $N\subseteq E^*$ are defined as follows
$M^{\circ}:=\{y\in E^*:|y(x)|\leq 1\text{ for all } x\in M\}$ and
$N^{\circ}:=\{x\in E:\;|y(x)|\leq 1\text{ for all } y\in N\}.$
If $A$ is a linear operator
acting on $E,$
then the domain, kernel space and range of $A$ will be denoted by
$D(A),$ $N(A)$ and $R(A),$
respectively. Since no confusion
seems likely, we will identify $A$ with its graph. In the remaining part of this paragraph, we assume that the operator $A$ is closed. Set
$p_{A}(x):=p(x)+p(Ax),$ $x\in D(A),$ $p\in \circledast$.
Then the calibration $(p_{A})_{p\in \circledast}$ induces the Hausdorff sequentially complete locally convex topology on $D(A);$
we denote this space simply by $[D(A)].$
Set $D_{\infty}(A):=\bigcap_{n=1}^{\infty}D(A^n).$ Then the space $D_{\infty}(A),$ equipped with the following system of seminorms $p_{n}(x):=p(x)+p(Ax)+p(A^2x)+\cdot \cdot \cdot+p(A^nx),$ $x\in D_{\infty}(A)$ ($n \in {\mathbb N},$ $p\in \circledast$) becomes an SCLCS; we will denote this space by
$[D_{\infty}(A)].$ Clearly, if $A$ is a Fr\' echet space, then $[D_{\infty}(A)]$ is a Fr\' echet space, as well.
If $C\in L(E)$ is injective, then we define the
$C$-resolvent set of $A,$
$\rho_{C}(A)$ for short, by
\begin{equation}\label{C-res}
\rho_{C}(A):=\Bigl\{\lambda \in {\mathbb C} : \lambda -A \mbox{ is
injective and } (\lambda-A)^{-1}C\in L(E)\Bigr\}.
\end{equation}
By the closed graph theorem\index{Closed Graph Theorem} \cite{meise}, the following holds:
If $E$ is a webbed bornological space\index{space!webbed bornological}
(this, in particular, holds if $E$ is a Fr\' echet space\index{space!Fr\' echet}), then the
$C$-resolvent set\index{$C$-resolvent set} of $A$ consists of those complex numbers $\lambda$
for which the operator $\lambda -A$ is injective and $R(C)
\subseteq R(\lambda -A).$ The resolvent set of $A,$ denoted by $\rho(A),$ is nothing else but the
$I$-resolvent set of $A,$ where $I$ denotes the identity operator on $E.$ Unless stated otherwise,
we shall always assume that $CA\subseteq AC.$
By $\sigma_{p}(A),$ $\sigma_{c}(A)$ and $\sigma_{r}(A)$ we denote the point, continuous and residual spectrum of $A,$ respectively.
For a closed linear operator
$A$, we introduce the subset $A^*$ of $E^*\times E^*$ by
$$
A^{*}:=\Bigl\{\bigl(x^*, y^*\bigr)\in E^*\times E^*:x^*(Ax)=y^*(x)\text{ for all
}x\in D(A)\Bigr\}.
$$
If $A$ is densely defined, then $A^*$ is also known as the adjoint operator\index{operator!adjoint} of $A$ and it
is
a closed linear operator on $E^*$.

The exponential region $E(a,b)$ has been defined for the first time by W. Arendt, O. El--Mennaoui and V. Keyantuo in \cite{a22}:
$$
e(a,b):=\Bigl\{\lambda\in\mathbb{C}:\Re\lambda\geq b,\:|\Im\lambda|\leq e^{a\Re\lambda}\Bigr\} \ \ (a,\ b>0).
$$

\subsection{Structural properties}
Now we are going to explain the notions of various types of generalized function spaces used throughout the paper.
We begin with the recollection of the most important properties of vector-valued distribution spaces (cf. \cite{ant1},
\cite{a43}, \cite{fat1}, \cite{komatsu}, \cite{knjigah}, \cite{kothe1}, \cite{martinez}-\cite{meise}, \cite{stan},
 \cite{sch16}, \cite{yosida} and references cited therein
for the basic information in this direction).
The Schwartz spaces of test functions $\mathcal{D}=C_0^{\infty}(\mathbb{R})$
and $\mathcal{E}=C^{\infty}(\mathbb{R})$
carry the usual topologies.
The spaces
$\mathcal{D}'(E):=L(\mathcal{D},E)$,
$\mathcal{E}'(E):=L(\mathcal{E},E)$ and
$\mathcal{S}'(E):=L(\mathcal{S},E)$
are topologized in the very obvious way;
$\mathcal{D}'_{\Omega}(E)$,
$\mathcal{E}'_{\Omega}(E)$ and $\mathcal{S}'_{\Omega}(E)$ denote the subspaces of
$\mathcal{D}'(E)$, $\mathcal{E}'(E)$ and $\mathcal{S}'(E)$,
respectively, containing $E$-valued distributions
whose supports are contained in $\Omega ;$ $\mathcal{D}'_{0}(E)\equiv \mathcal{D}'_{[0,\infty)}(E)$, $\mathcal{E}'_{0}(E)\equiv \mathcal{E}'_{[0,\infty)}(E)$, $\mathcal{S}'_{0}(E)\equiv \mathcal{S}'_{[0,\infty)}(E).$
In the case that $E={\mathbb C},$ then the above spaces are also denoted by $\mathcal{D}',$
$\mathcal{E}',$
$\mathcal{S}',$ $\mathcal{D}'_{\Omega},$
$\mathcal{E}'_{\Omega},$
$\mathcal{S}'_{\Omega},$
$\mathcal{D}_0'$,
$\mathcal{E}_0'$ and $\mathcal{S}_0'.$
If $\varphi$, $\psi:\mathbb{R}\to\mathbb{C}$ are
measurable functions, then we define the convolution products $\varphi*\psi$
and $\varphi*_0\psi$ by
$$
\varphi*\psi(t):=\int\limits_{-\infty}^{\infty}\varphi(t-s)\psi(s)\,ds\mbox{ and }
\varphi*_0
\psi(t):=\int\limits^t_0\varphi(t-s)\psi(s)\,ds,\;t\;\in\mathbb{R}.
$$
Notice that $\varphi*\psi=\varphi*_0\psi$, provided that supp$(\varphi)$ and supp$(\psi)$ are subsets of $[0,\infty).$
Given $\varphi\in\mathcal{D}$ and $f\in\mathcal{D}'$, or $\varphi\in\mathcal{E}$ and $f\in\mathcal{E}'$,
we define the convolution $f*\varphi$ by $(f*\varphi)(t):=f(\varphi(t-\cdot))$, $t\in\mathbb{R}$.
For $f\in\mathcal{D}'$, or for $f\in\mathcal{E}'$,
define $\check{f}$ by $\check{f}(\varphi):=f(\varphi (-\cdot))$, $\varphi\in\mathcal{D}$ ($\varphi\in\mathcal{E}$).
Generally, the convolution of two distribution $f$, $g\in\mathcal{D}'$, denoted by $f*g$,
is defined by $(f*g)(\varphi):=g(\check{f}*\varphi)$, $\varphi\in\mathcal{D}$.
Then we know that $f*g\in\mathcal{D}'$ and supp$(f*g)\subseteq$supp$ (f)+$supp$(g)$.

Let $G$ be an $E$-valued distribution, and let $f : {\mathbb R} \rightarrow E$ be a locally integrable function (cf. \cite[Definition 1.1.4, Definition 1.1.5]{knjigaho}).
As in the scalar-valued case, we define the $E$-valued distributions
$G^{(n)}$ ($n\in {\mathbb N}$) and $ hG$ ($h\in {\mathcal E}$); the regular $E$-valued distribution ${\mathbf f}$ is defined by ${\mathbf f}(\varphi):=\int_{-\infty}^{\infty}\varphi (t)  f(t) \, dt$ ($\varphi \in {\mathcal D}$).
We need the following auxiliary lemma whose proof can be deduced as in the scalar-valued case (\cite{stan}).

\begin{lem}\label{polinomi}
Suppose that $0<\tau \leq \infty,$ $n\in {\mathbb N}$. If $f : (0,\tau) \rightarrow E$ is a continuous function and
$$
\int \limits^{\tau}_{0}\varphi^{(n)}(t)f(t)\, dt=0,\quad \varphi \in {\mathcal D}_{(0,\tau)},
$$
then there exist elements $x_{0},\cdot \cdot \cdot, x_{n-1}$ in $E$ such that $f(t)=\sum^{n-1}_{j=0}t^{j}x_{j},$ $t\in (0,\tau).$
\end{lem}

Let $\tau>0,$ and let $X$ be a general Hausdorff locally convex space (not necessarily sequentially complete). Following L. Schwartz \cite{sch166}, it will be said that a distribution $G\in {\mathcal D}'(X)$ is of finite order on the interval $(-\tau,\tau)$ iff
there exist an integer $n\in {\mathbb N}_{0}$ and an $X$-valued continuous function $f : [-\tau,\tau] \rightarrow X$
such that
$$
G(\varphi)=(-1)^{n}\int^{\tau}_{-\tau}\varphi^{(n)}(t)f(t)\, dt,\quad \varphi \in {\mathcal D}_{(-\tau,\tau)}.
$$

Let us recall that the spaces of Beurling,
respectively, Roumieu ultradifferentiable functions are
defined by $\mathcal{D}^{(M_p)}:=\mathcal{D}^{(M_p)}(\mathbb{R})
:=\text{indlim}_{K\Subset\Subset\mathbb{R}}\mathcal{D}^{(M_p)}_K$,
respectively,
$\mathcal{D}^{\{M_p\}}:=\mathcal{D}^{\{M_p\}}(\mathbb{R})
:=\text{indlim}_{K\Subset\Subset\mathbb{R}}\mathcal{D}^{\{M_p\}}_K$, (where $K$ goes through all compact sets in ${\mathbb R}$
where
$\mathcal{D}^{(M_p)}_K:=\text{projlim}_{h\to\infty}\mathcal{D}^{M_p,h}_K$,
respectively, $\mathcal{D}^{\{M_p\}}_K:=\text{indlim}_{h\to 0}\mathcal{D}^{M_p,h}_K$,
\begin{align*}
\mathcal{D}^{M_p,h}_K:=\bigl\{\phi\in C^{\infty}(\mathbb{R}): \text{supp}(\phi) \subseteq K,\;\|\phi\|_{M_p,h,K}<\infty\bigr\},
\end{align*}

\begin{align*}
\|\phi\|_{M_p,h,K}:=\sup\Biggl\{\frac{h^p\bigl|\phi^{(p)}(t)\bigr|}{M_p} : t\in K,\;p\in\mathbb{N}_0\Biggr\}.
\end{align*}
Henceforth the asterisk $*$ stands for both cases. Let $\emptyset \neq \Omega \subseteq {\mathbb R}.$
The spaces
$\mathcal{D}'^*(E):=L(\mathcal{D}^*, E)$, $\mathcal{D}^{*}_{\Omega}$, $\mathcal{D}^{\ast}_0$, $\mathcal{E}'^{*}_{\Omega}$, $\mathcal{E}'^{*}_{0}$, $\mathcal{D}'^{*}_{\Omega}(E)$ and $\mathcal{D}'^{*}_{0}(E)$ are defined as in the case of Schwartz spaces.
An entire function of the form
$P(\lambda)=\sum_{p=0}^{\infty}a_p\lambda^p$,
$\lambda\in\mathbb{C}$ is of class $(M_p)$, respectively, of
class $\{M_p\}$, if there exist $l>0$ and $C>0$, respectively, for every
$l>0$ there exists a constant $C>0$, such that $|a_p|\leq Cl^p/M_p$,
$p\in\mathbb{N};$ cf. \cite{k91} for further information.
The corresponding ultradifferential operator
$P(D)=\sum_{p=0}^{\infty}a_p D^p$ is of class $(M_p)$, respectively, of
class $\{M_p\}$. Since $(M_{p})$ satisfies (M.2), the ultradifferential operator $P(D)$ of $\ast$-class
$$
\langle P(D)G,\varphi \rangle :=\langle G, P(-D)\varphi \rangle,\quad G\in \mathcal{D}'^*(E),\ \varphi \in \mathcal{D}^*,
$$
is a continuous linear mapping from $\mathcal{D}'^*(E)$ into $\mathcal{D}'^*(E).$ The multiplication by a function $a\in {\mathcal E}^{\ast}(\Omega)$,
convolution of scalar valued ultradistributions (ultradifferentiable
functions), and the notion of a regularizing sequence in $\mathcal{D}^*,$ are defined as in the case of
distributions;
we know that there exists
a regularizing sequence in $\mathcal{D}$ ($\mathcal{D}^*$). If $\varphi\in\mathcal{D}^{\ast}$ ($T\in {\mathcal E}'^{\ast}$) and $G\in\mathcal{D}'^{\ast}(E)$, then
$\varphi \ast G \in {\mathcal E}^{\prime \ast}(E)$ and $T \ast G \in  {\mathcal D}^{\prime \ast}(E)$ (cf. \cite[p. 685]{k82}, and
\cite[Definition 3.9]{k82} for the notion of space ${\mathcal E}^{\prime \ast}(E)$).

Following \cite[Definition 4.5]{k82}, we say that a vector-valued ultradistribution $G\in\mathcal{D}'^*(E)$ is bounded iff it maps any neighborhood of zero in $\mathcal{D}^*$ into a bounded subset of $E.$ A vector-valued ultradistribution $G\in\mathcal{D}'^*(E)$ is
said to be locally bounded iff for every compact subset $K$ of $\Omega ,$ $G$ maps any neighborhood of zero in $\mathcal{D}^{*}_{K}$ into a bounded subset of $E.$ Recall that any vector-valued ultradistribution $G\in\mathcal{D}'^*(E)$ is bounded (locally bounded)
if $E$ is metrizable (if $E$ is a (DF) space).

\begin{thm}\label{1.3.1.4} (\cite[Theorem 4.6, Theorem 4.7]{k82})
 \emph{(i)}  Let $G\in\mathcal{D}'^*(E)$ be locally bounded.
Then, for each relatively compact non-empty open set $\Omega\subseteq\mathbb{R},$
there exists a sequence of continuous function $(f_n)$ in $E^{\overline{\Omega}}$ such that
$$
G_{|\Omega}=\sum_{n=0}^{\infty}D^nf_n
$$
and there exists $L>0$ in the Beurling case, resp.,
for every $L>0$ in the Roumieu case,
such that the set $\{M_{n}L^{n}f_n(t) : t\in\overline{\Omega},\ n\in\mathbb{N}\}$ is bounded in $E.$

 \emph{(ii)} Let $G\in\mathcal{D}'^*(E)$ be locally bounded. Suppose, additionally, that $(M_p)$ satisfies \emph{(M.3)}.
Then for each relatively compact non-empty open set $\Omega\subseteq\mathbb{R}$
there exist an ultradifferential operator of $*$-class and a continuous function $f:\overline{\Omega}\to E$
such that $G_{|\Omega}=P(D)f$.
\end{thm}


\begin{thm}\label{delta}
(\cite[Theorem 4.8]{k82}) Suppose that $(M_p)$ additionally satisfies \emph{(M.3)} as well as that $G\in\mathcal{D}'^*(E)$
and \emph{supp}$(G)\subseteq\{0\}$.
Then there exists a sequence $(x_n)_{n\in {{\mathbb N}_{0}}}$ in $E$ such that $G(\varphi)=\sum_{n=0}^{\infty}\delta^{(n)}(\varphi)x_n$,
$\varphi\in\mathcal{D}^*$
and there exists $L>0$ in the Beurling case,
resp., for every $L>0$ in the Roumieu case,
such that the set $\{M_{n}L^{n}x_n :  n\in\mathbb{N}\}$ is bounded in $E.$
\end{thm}

The characterization of vector-valued distributions supported by a point has been studied by R. Shiraishi, Y. Hirata \cite{1964} and T. Ushijima \cite{ush1}. Their results can be briefly described as follows: Suppose that $G\in\mathcal{D}'(E)$ and supp$(G)\subseteq\{0\}$. If there exists a norm $\| \cdot \|$ on $E$ satisfying that $\|x\| \leq cp(x),$ $x\in E$ for some $p\in \circledast$ and $c>0,$ or if $E$ satisfies the conditions ($\ast$)-($\ast$)' stated on page 82 of \cite{1964}, then there exist $n\in\mathbb{N}$ and $x_i\in E$, $0\leq i\leq n$
such that $G(\varphi)=\sum_{i=0}^n\delta^{(i)}(\varphi)x_i$, $\varphi\in\mathcal{D}$. If the space $E$ satisfies the property that any vector-valued distribution $G\in\mathcal{D}'(E)$ with supp$(G)\subseteq\{0\}$ can be represented as a finite sum of
vector-valued distributions like $\delta^{(i)} \otimes x_{i}$ (cf. the next paragraph for the notion), then we
shall simply say that $E$ is admissible.

We need the following basic facts about the Laplace transform of (ultra-)distributions (cf. T. K\=omura
\cite{komura}). Let
$$
\hat{\varphi}(\lambda):=\frac{1}{2\pi} \int \limits^{\infty}_{-\infty}e^{\lambda t}\varphi(t)\, dt,\quad \lambda \in {\mathbb C},\ \varphi \in {\mathcal D} \ \ (\varphi \in {\mathcal D}^{\ast}).
$$
Set ${\mathbf D}:=\{ \hat{\varphi} : \varphi \in {\mathcal D}\}$ and ${\mathbf D}^{\ast}:=\{ \hat{\varphi} : \varphi \in {\mathcal D}^{\ast}\},$ $\hat{\varphi} +\hat{\psi}:=\widehat{\varphi +\psi},$ $\lambda \hat{\varphi}:=\hat{\lambda \varphi}$ ($\lambda \in {\mathbb C};$ $\varphi,$ $\psi \in {\mathcal D}$ (${\mathcal D}^{\ast}$)).
Then the mapping $\hat \ : {\mathcal D} \rightarrow {\mathbf D}$ ($\hat \ : {\mathcal D}^{\ast} \rightarrow {\mathbf D}^{\ast}$)
is a linear isomorphism between the vector spaces ${\mathcal D}$ and ${\mathbf D}$ (${\mathcal D}^{\ast}$ and ${\mathbf D}^{\ast}$). It is said that a subset ${\mathbf T}=\{\hat{\varphi} : \varphi \in {\mathcal T}\}$ of ${\mathbf D}$ (${\mathbf D}^{\ast}$)
is open iff the set ${\mathcal T}$ is open in ${\mathcal D}$ (${\mathcal D}^{\ast}$).
Then the mapping $\hat{\cdot}$ becomes a linear topological homeomorphism between the Hausdorff locally convex spaces ${\mathcal D}$ (${\mathcal D}^{\ast}$) and ${\mathbf D}$ (${\mathbf D}^{\ast}$). Set ${\mathbf D}'(E):=L({\mathbf D},E)$ and ${\mathbf D}'^{*}(E):=L({\mathbf D}^{\ast},E).$ There exists a linear topological homeomorphism $\hat{\cdot} : {\mathcal D}'(E) \rightarrow {\mathbf D}'(E)$
($\hat{\cdot} : {\mathcal D}'^{\ast}(E) \rightarrow {\mathbf D}'^{\ast}(E)$) defined by $\hat{G}(\hat{\varphi}):=G(\varphi),$ $\varphi \in {\mathcal D}$ ($\hat{G}(\hat{\varphi}):=G(\varphi),$ $\varphi \in {\mathcal D}^{\ast}$) for all
$G\in {\mathcal D}'(E)$ ($G\in{\mathcal D}'^{\ast}(E)$).
The functional $\hat{G}$ is called the generalized Laplace transform of $G.$ Set, for every non-empty subset $\Omega$ of ${\mathbb R},$
${\mathbf D}'_{\Omega}(E):=\{\hat{G} : G\in {\mathcal D}'_{\Omega}(E)\}$ (${\mathbf D}'^{*}_{\Omega}(E):=\{\hat{G} : G\in {\mathcal D}'^{*}_{\Omega}(E)\}$);
${\mathbf D}'_{0}(E)\equiv {\mathbf D}'_{[0,\infty)}(E)$ (${\mathbf D}'^{*}_{0}(E) \equiv {\mathbf D}'^{*}_{[0,\infty)}(E)$). Then ${\mathbf D}'_{0}(E)$ (${\mathbf D}'^{*}_{0}(E)$) is a closed subspace of ${\mathbf D}'(E)$ (${\mathbf D}'^{*}(E)$) and it is topologically homeomorphic to
${\mathcal D}'_{0}(E)$ (${\mathcal D}'^{*}_{0}(E)$)
by the mapping $\hat{\cdot}.$ If $F\in {\mathcal D}'$ ($F\in {\mathcal D}'^{\ast}$) and $x\in E$, then we define $F \otimes x \in {\mathcal D}'(E)$ ($F \otimes x \in {\mathcal D}'^{\ast}(E)$) and $\hat{F} \otimes x \in {\mathbf D}'(E)$ ($\hat{F} \otimes x \in {\mathbf D}'^{\ast}(E)$)
by
$\langle F \otimes x ,\varphi \rangle:=\langle F , \varphi \rangle x,$
$\varphi \in {\mathcal D}$ ($\varphi \in {\mathcal D}^{\ast}$)
and
$\langle \hat{F} \otimes x ,\hat{\varphi} \rangle :=\langle F, \varphi \rangle x,$ $\varphi \in {\mathcal D}$ ($\varphi \in {\mathcal D}^{\ast}$).

If $T \in {\mathcal E}'(E)$ ($T \in {\mathcal E}'^{\ast}(E)$), then we define the Laplace transform of $T$ by
$$
\hat{T}(\lambda):=\bigl \langle T(x), e^{\lambda x} \bigr \rangle,\quad \lambda \in {\mathbb C}.
$$

\begin{thm}\label{p-w-dis} (\cite{komura})
Let $k>0.$ Then
an $E$-valued entire function $f(\lambda)$ is the Laplace transform of a distribution $T\in {\mathcal D}'_{[-k,k]}(E)$
iff for every $x^{\ast}\in E^{\ast}$ there exist $n\in {\mathbb N}$ and $c>0$ such that
$$
\bigl | \bigl \langle x^{\ast}, f(\lambda) \bigr \rangle \bigr|\leq c(1+|\lambda|)^{n}e^{k|\Re \lambda|},\quad \lambda \in {\mathbb C}.
$$
\end{thm}

The proof of following Paley-Wiener theorem for $E$-valued ultradistributions with compact support can be deduced with the help of the corresponding assertion for scalar-valued ultradistributions \cite[Theorem 1.1]{k911} and the idea from \cite{komura}. This result can be viewed of some independent interest;
we will include all relevant details of proof for the sake of completeness.

\begin{thm}\label{EvPWU} (\emph{The Paley-Wiener theorem for $E$-valued ultradistributions}) Let \emph{(M.1)}, \emph{(M.2)} and \emph{(M.3)} hold, let ${\mathcal R}$ denote the set consisting of all positive monotonically increasing sequences, and let
$$
M_{r_p}(\rho)=\sup_{p\in {\mathbb N}}\Biggl\{\ln \frac{{\rho}^p}{M_p\prod_{i=1}^{p}r_i}\Biggr\},\quad \rho>0.
$$
An $E$-valued entire function $\hat{u}(\lambda)$ is the generalized Laplace transform of an $E$-valued ultradistribution $u$ of $\ast$-class with support contained in a non-empty compact subset $K\subseteq {\mathbb R}$ iff for every $x^{\ast}\in E^{\ast}$ there exist $h>0$ and $c>0$, in Beurling case, resp., there exist $(r_p)\in{\mathcal R}$ and $c_{r_p}>0,$ in Roumieu case, such that
\begin{align}
\notag \bigl |  & \bigl  \langle x^{\ast}, \hat{u}(\lambda)  \bigr  \rangle \bigr|\leq c e^{M(\lambda/h)+H_K(i\lambda)},\quad \lambda\in{\mathbb C},\mbox{ resp.,}
\\ \label{E2} &
\bigl|  \bigl  \langle x^{\ast}, \hat{u}(\lambda)  \bigr  \rangle \bigr|\leq c_{r_p} e^{M_{r_p}(\lambda)+H_K(i\lambda)},\quad \lambda\in{\mathbb C}.
\end{align}
Here $H_{K}(\lambda):=\sup_{x\in K} \Im (x\lambda),$ $\lambda \in {\mathbb C}.$
\end{thm}

\begin{proof}
The proof of necessity follows almost immediately from the Paley-Wiener theorem for scalar-valued ultradistributions (see \cite[Theorem 1.1]{k911}). In the proof of sufficiency,
we will consider only the Roumieu case.
So, let us assume that $\hat{u}(\lambda)$ is an $E$-valued entire function satisfying (\ref{E2}). It suffices to show (see \cite[Example, p. 267]{komura} and \cite[Lemma 3.3]{k91}) that $\hat{u}(\lambda)$ satisfies: For any continuous seminorm $p$ on $E$, there exist $({r_p})\in{\mathcal R}$ and $c_{r_p}>0$ such that
\begin{equation}\label{ponom}
p\bigl(\hat{u}(\lambda)\bigr)\leq c_{r_p} e^{M_{r_p}(\lambda)+H_K(i\lambda)}, \quad \lambda\in{\mathbb C}.
\end{equation}
Let us suppose the converse, i.e., that (\ref{E2}) holds but (\ref{ponom}) does not hold. We will construct sequences $(r_n)$, $(c_n)$, $(\varepsilon_n)$, $(\lambda_n)$ and $(x_{n}^{\ast})$ satisfying certain properties. Set $\varepsilon_{1}:=1/2.$ It is clear that there exist $\lambda_1\in{\mathbb C}$, $r_1>0$ and $c_1>0$ such that, for some continuous seminorm $q$ on $E$, one has
$$
q\bigl(\hat{u}(\lambda_1)\bigr)>c_{r_1}e^{M_{r_1}(\lambda_1)+H_K(i\lambda_1)}.
$$
Observe also that $q(x)=\sup_{x^{\ast}\in U^{\circ}}|\langle x^{\ast},x\rangle|$ for all $x\in E,$
where $U=\{x\in E\, :\, q(x)\leq1\}.$ Choose $x_{1}^{\ast}\in U^{\circ}$ such that $|\langle x_{1}^{\ast},\hat{u}(\lambda)\rangle|>c_{r_1}e^{M_{r_1}(\lambda_1)+H_K(i\lambda_1)}$. Suppose that $(r_i)$, $(C_i)$, $(\varepsilon_i)$, $(\lambda_i)$ and $(x_{i}^{\ast})$ are determined for $1\leq i\leq N-1.$ Then there exist $c_{r_{N}}>c_{r_{N-1}}+1$ and $r_{N}>r_{N-1}+1$ with
\begin{equation}\label{zvpar}
\bigl| \bigl \langle x_{N-1}^{\ast},\hat{u}(\lambda) \bigr \rangle \bigr|\leq c_{r_N} e^{M_{r_N}(\lambda)+H_K(i\lambda)},\quad \lambda \in {\mathbb C}.
\end{equation}
Having in mind that the set $\{|\langle x^{\ast}, {\hat{u}}(\lambda_i)\rangle|\, :\,\, 1\leq i\leq N-1,\ x^{\ast}\in U^{\circ}\}$ is bounded, we obtain that there exists $\varepsilon_N\leq1/2^N$ such that
$$
\sup\limits_{1\leq i\leq N-1, x^{\ast}\in U^{\circ}}\bigl| \bigl \langle x^{\ast},\hat{u}(\lambda_i) \bigr \rangle \bigr |\leq\frac{1}{2^N\varepsilon_N}.
$$
Now, by the assumption, there exists $\lambda_{N} \in {\mathbb C}$ such that
$$
q\bigl(\hat{u}(\lambda_N)\bigr)>\frac{3C_{r_N}}{\varepsilon_N}e^{M_{r_N}(\lambda_N)+H_K(i\lambda_N)}.
$$
There is an $x_{N}^{\ast}\in U^{\circ}$ such that
$$
\bigl| \bigr\langle\hat{u}(\lambda_N),x_{N}^{\ast} \bigr\rangle \bigr|>\frac{3C_{r_N}}{\varepsilon_N}e^{M_{r_N}(\lambda_N)+H_K(i\lambda_N)}.
$$
Set $x_{\infty}^{\ast}:=\sum_{n=1}^{\infty}\varepsilon_n x_{n}^{\ast}$. Since $U^{\circ}$ is convex, balanced and $\sigma$-compact, $\sum_{n=1}^{\infty}\varepsilon_n\leq 1$, we have $x_{\infty}^{\ast}\in U^{\circ}$. Hence, by (\ref{zvpar}),
\begin{align*}
\Bigl| \bigl \langle\hat{u}&(\lambda_N),x_{\infty}^{\ast} \bigr \rangle\Bigr|=\Biggl|\sum\limits_{n=1}^{\infty}\bigl \langle\hat{u}(\lambda_N),\varepsilon_nx_{n}^{\ast} \bigr \rangle \Biggr|
 \\ & \geq \Bigl| \bigl \langle\hat{u}(\lambda_N),\varepsilon_Nx_{N}^{\ast} \bigr \rangle\Bigr|-\sum\limits_{n=1}^{N-1}\Bigl|  \bigl \langle\hat{u}(\lambda_N), \varepsilon_nx_{n}^{\ast} \bigr \rangle \Bigr|-\sum\limits_{n=N+1}^{\infty}\Bigl|  \bigl \langle\hat{u}(\lambda_N),\varepsilon_nx_{n}^{
\ast}  \bigr \rangle \Bigr|
\\ & > 3C_{r_N}e^{M_{r_N}(\lambda_N)+H_K(i\lambda_N)}-\sum\limits_{n=1}^{N-1}\Bigl|\varepsilon_n e^{M_{r_N}(\lambda_N)+H_K(i\lambda_N)}\Bigr |
\\ & -\sum\limits_{n=N+1}^{\infty}\frac{1}{2^n}>C_{r_N}e^{M_{r_N}(\lambda_N)+H_K(i\lambda_N)},
\end{align*}
which contradicts (\ref{E2}).
\end{proof}

For further information concerning the Paley-Wiener type theorems for ultradifferentiable functions and
infinitely differentiable functions with compact support, we refer the reader to
\cite[Section 9]{k91}, \cite{yosida} and \cite[Section 11.6]{stan}.

The spaces of tempered ultradistributions of the Beurling,
resp., the Roumieu type, are defined in \cite{pilip} as duals of the following test spaces
$$
\mathcal{S}^{(M_p)}(\mathbb{R}^{n}):=\text{projlim}_{h\to\infty}\mathcal{S}^{M_p,h}(\mathbb{R}^{n}),
\mbox{ resp., }\mathcal{S}^{\{M_p\}}(\mathbb{R}^{n}):=\text{indlim}_{h\to 0}\mathcal{S}^{M_p,h}(\mathbb{R}^{n}),
$$
where for each $h>0,$
\begin{align*}
\mathcal{S}^{M_p,h}(\mathbb{R}^{n}):=\bigl\{\phi\in C^\infty(\mathbb{R}^{n}):\|\phi\|_{M_p,h}<\infty\bigr\},
\end{align*}
\begin{align*}
\|\phi\|_{M_p,h}:=\sup\Biggl\{\frac{h^{|\alpha|+|\beta|}}{M_{|\alpha|} M_{|\beta|}}\bigl(1+|x|^2\bigr)^{\beta/2}\bigl|\phi^{(\alpha)}(x)\bigr| : x\in\mathbb{R}^{n},
\;\alpha,\;\beta\in\mathbb{N}_{0}^{n}\Biggr\}.
\end{align*}
If $n=1,$
then we also write $\mathcal{S}^{(M_p)}$ ($\mathcal{S}^{\{M_p\}}$) for $\mathcal{S}^{(M_p)}(\mathbb{R}^{n})$ ($\mathcal{S}^{\{M_p\}}(\mathbb{R}^{n})$); the common abbreviation for the both case of brackets will be $\mathcal{S}^{\ast}.$
For further information we refer to \cite{b42}-\cite{ckm}, \cite{cizi}, \cite{fat}, \cite{gr}, \cite{knjigah},
\cite{dusanka}, \cite{ku113}, \cite{me152}, \cite{patak} and \cite{pilip}.

\section{$C$-wellposedness of first order Cauchy problem in the sense of distributions and ultradistributions}

In this section, we will continue the study of T. Ushijima
\cite[Section 1-Section 2]{ush1} on the well-posedness of Cauchy problem in the spaces of abstract distributions. We note that there exist some assertions in the existing literature on abstract differential equations in locally convex spaces, like \cite[Proposition 1.2]{komura} or
\cite[Propositions 1.1, 1.3-1.4, 1.6; Theorem 2.1]{ush1}, in which the sequential completeness of the state space $E$ has not been assumed. We will not follow this general approach here.

\begin{defn}\label{C-wellposed} (cf. also \cite[Definition 2.1.4]{me152} for distribution case, with $C=I$)
Let $A$ be a closed linear operator on $E,$ let $C\in L(E)$ be injective, and let $CA\subseteq AC.$ Then it is said that the operator $A$ is $C$-wellposed for the abstract Cauchy problem $u'-Au=G$ at $t=0$ in the sense of distributions
(ultradistributions of $\ast$-class)
if for each $G\in {\mathcal D}'_{0}(E)$ ($G\in {\mathcal D}'^{\ast}_{0}(E)$) there exists a unique $U_{G}\in {\mathcal D}'_{0}(E)$ ($U_{G} \in {\mathcal D}'^{\ast}_{0}(E)$)  satisfying the following conditions:
\begin{itemize}
\item[(i)] $U_{G}(\varphi) \in D(A)$ for all $\varphi \in {\mathcal D}$ ($\varphi \in {\mathcal D}^{\ast}$),
\item[(ii)] the mapping $G \mapsto U_{G},$ $G\in {\mathcal D}'_{0}(E)$ ($G\in {\mathcal D}'^{\ast}_{0}(E)$) belongs to $L( {\mathcal D}'_{0}(E))$ ($L({\mathcal D}'^{\ast}_{0}(E))$),
\item[(iii)] $U_{G}'(\varphi)-AU_{G}(\varphi)=CG(\varphi)$ for all $\varphi \in {\mathcal D}$ ($\varphi \in {\mathcal D}^{\ast}$).
\end{itemize}
\end{defn}

\begin{defn}\label{exp C-wellposed}  (cf. also \cite[Subsection 2.1.3]{me152} for distribution case, with $C=I$)
 Let $A$ be a closed linear operator on $E$, let $C\in L(E)$ be injective, and let $CA\subseteq AC$. Then  it is said that the operator $A$ is exponentially $C$-wellposed for the abstract Cauchy problem $u'-Au=G$ at $t=0$ in sense of distributions (ultradistributions of $\ast$-class) if for each $G\in\DD'_0(E)$ ($G\in {\mathcal D}'^{\ast}_{0}(E)$) there exists a unique $U_G\in\DD'_0(E)$ ($U_{G} \in {\mathcal D}'^{\ast}_{0}(E)$) satisfying (i), (ii) and (iii) from the previous definition and the following condition:
\begin{itemize}
\item[(iv)] there exists $a\geq 0$ such that $e^{-a\cdot}U_G\in\SSS'(E)$ ($e^{-a\cdot}U_G\in  {\mathcal S}'^{\ast}(E)$).
\end{itemize}
\end{defn}


\subsection{$C$-generalized resolvents of linear operators}

Throughout this subsection, we assume that $X$ and $Y$ are Hausdorff locally convex spaces over the field ${\mathbb K} \in \{{\mathbb R},{\mathbb C}\}$ as well as that $A$ is a linear operator on $Y$ and the operator $C\in L(Y)$ is injective.
Let $Z$ be a non-trivial subspace of $L(X,Y)$ obeying the property that $CU\in Z$ whenever $U\in Z,$ and let ${\mathbf B}$ denote the family of all bounded subsets of $X.$ By $I_{X}$ ($I_{Y},$ $I_{Z}$) we denote the identity operator on $X$ ($Y$, $Z$). Then $Z$ is a
Hausdorff locally convex space over the field ${\mathbb K},$ and the fundamental system of seminorms which defines the topology of $Z$
is $(P_{B})_{P\in \circledast_{Y},B\in {\mathbf B}},$ where $P_{B}(T):=\sup_{x\in B}P(Tx),$ $T\in Z$ ($P\in \circledast_{Y},$ $B\in {\mathbf B}$). In \cite[Definition 4.1]{komura}, T. K\=omura has analyzed the case $X={\mathcal D}_{(-\infty ,a]},$ $Z=L({\mathcal D}_{(-\infty ,a]},Y)$ for some $a>0,$ and $C=I_{Y},$ while on pages 96-97 of \cite{ush1} T. Ushijima has analyzed the case $X={\mathcal D},$ $Z={\mathcal D}'_{0}(Y)$ and $C=I_{Y}.$

The main aim of this subsection is to provide, on the basis of ideas from \cite{komura} and \cite{ush1} (cf. also \cite[Definition 1]{vuaq}), a very general approach for introducing the notions of $C$-generalized resolvents of linear operators. Define a linear operator $A_{X,Z}$ on $Z$ by
$$
A_{X,Z}:=\bigl\{(U,V) \in Z \times Z : Ux\in D(A)\mbox{ for all }x\in X\mbox{ and }Vx=A(Ux),\ x\in X \bigr\}.
$$
Then it is checked at once that $D_{X,Z}\in L(Z)$ for any $D\in L(Y)$ satisfying that $DU\in Z$ for all $U\in Z,$ and that the operator $C_{X,Z}\in L(Z)$ is injective. Furthermore, the assumption $CA\subseteq AC$ ($C^{-1}AC=A$)
implies
$C_{X,Z}A_{X,Z}\subseteq A_{X,Z}C_{X,Z}$ ($C_{X,Z}^{-1}A_{X,Z}C_{X,Z}=A_{X,Z}$).
If the closed graph theorem holds for the mappings from $X$ into $Y,$ then $D(A_{X,Z})$ consists exactly of those mappings $U\in Z$ for which $R(U)\subseteq D(A)$ and $AU\in Z;$ in this case, $A_{X,Z}=AU$ for all $U\in D(A_{X,Z}).$

\begin{defn}\label{gen-res}
The $C_{X,Z}$-resolvent set of $A,$ $\rho_{C_{X,Z}}(A)$ in short, is defined as the set of those scalars $\lambda \in {\mathbb K}$
for which  the operator $\lambda I_{Z}-A_{X,Z}$ is injective, $R(C_{X,Z})\subseteq R(\lambda I_{Z}-A_{X,Z})$ and
$(\lambda I_{Z}-A_{X,Z})^{-1}C_{X,Z}\in L(Z).$ If $C=I_{Y},$ then the $C_{X,Z}$-resolvent set of $A$
is also called the $_{X,Z}$-resolvent set of $A$ and
denoted by
$\rho_{_{X,Z}}(A)$ for short.
\end{defn}

In other words, the $C_{X,Z}$-resolvent set of $A$ (the $_{X,Z}$-resolvent set of $A$) is defined as the $C_{X,Z}$-resolvent set (the
resolvent set) of the operator $A_{X,Z}$ in $Z.$
The $C_{X,Z}$-spectrum of $A,$ denoted by $\sigma_{C_{X,Z}}(A),$ is defined as the complement of set $\rho_{C_{X,Z}}(A)$ in ${\mathbb K};$ in the case that $C=I_{Y},$  $\sigma_{C_{X,Z}}(A)$ is also denoted by $\sigma_{_{X,Z}}(A)$ and called
the $_{X,Z}$-spectrum of $A.$ We can decompose the $_{X,Z}$-spectrum of $A$ into three disjunct subsets:
\begin{itemize}
\item[(i)] the point $_{X,Z}$-spectrum of $A$, shortly $\sigma_{p;_{X,Z}}(A),$ consisting of the eigenvalues of the operator $A_{X,Z},$
\item[(ii)] the continuous $_{X,Z}$-spectrum of $A$, shortly $\sigma_{c;_{X,Z}}(A),$ consisting of the scalars that are not eigenvalues of the operator $A_{X,Z},$ but make the range of $\lambda I_{Z}-A_{X,Z}$ a proper dense subset of the space $Z$,
\item[(iii)] the residual  $_{X,Z}$-spectrum of $A$, shortly $\sigma_{r;_{X,Z}}(A),$ consisting of all other scalars in the spectrum.
\end{itemize}

It can be simply verified that the closedness (closability, injectivity) of the operator $A$ on $Y$ implies the closedness (closability, injectivity) of the
operator $A_{X,Z}$ on $Z$ (cf. also \cite[Proposition 4.3]{komura}), and that $\sigma_{p;_{X,Z}}(A)\subseteq \sigma_{p}(A)$. Suppose now that $\lambda \in \rho_{C}(A)$ (defined in the same way as in (\ref{C-res}), with ${\mathbb C}$ and $E$ replaced respectively by ${\mathbb K}$ and $Y$) and $(\lambda-A)^{-1}CU\in Z$ for all $U\in Z.$ Then $\lambda \in \rho_{C_{X,Z}}(A_{X,Z})$ and $(\lambda I_{Z}-A_{X,Z})^{-1}C_{X,Z}=((\lambda-A)^{-1}C)_{X,Z};$ in particular, $\rho_{C}(A)\subseteq
\rho_{C_{X,Z}}(A_{X,Z})$ provided that $Z=L(X,Y).$ In both approaches, T. K\=omura's or T. Ushijima's, the denseness of the operator $A$ implies the denseness of the operator $
A_{X,Z};$ using the proof of \cite[Proposition 4.4]{komura} and the consideration given on page 685 of \cite{k82}, it can be proved that the same conclusion
holds in the case that
$X={\mathcal D}^{\ast}_{(-\infty ,a]},$ $Z=L({\mathcal D}^{\ast}_{(-\infty ,a]},Y),$ for some $a>0$ and $C=I_{Y}$ (cf. \cite{komura}) or that $X={\mathcal D}^{\ast},$
$Z={\mathcal D}^{\prime \ast}_{0}(Y)$ and $C=I_{Y}$ (cf. \cite{ush1}). The following trivial example shows that the denseness of the operator $A$ does not imply the denseness of the operator $
A_{X,Z}$ in general case, as well as that the choice of spaces $X$ and $Z$ is very important for saying anything relevant and noteworthy about the operator $A_{X,Z}.$

\begin{example}\label{denseness}
Let $A$ be a densely defined linear operator on $Y,$ $C=I_{Y},$ let $U\in L(X,Y)$ satisfy that $R(U)$ is not contained in $D(A),$ and let $Z=\{\alpha U : \alpha \in {\mathbb K}\}.$ Then $D(A_{X,Z})=\{0\},$  and therefore, $A_{X,Z}$ is not densely defined in $Z.$
\end{example}

Apart from this, there exists a great number of well-known identities which continue to hold for $C$-generalized resolvents. For example, the validity of inclusion $CA\subseteq AC$ implies the following:

\begin{itemize}
\item[(a)] Let $k\in {\mathbb N}_{0}$ and $\lambda,\ z\in
\rho_{C_{X,Z}}(A_{X,Z})$ with $z\neq\lambda.$ Then
\begin{align*}
\notag \bigl(z&-A_{X,Z}\bigr)^{-1}C_{X,Z} \bigl(\lambda-A_{X,Z}\bigr)^{-k}
C_{X,Z}^{k}
\\ &=\frac{(-1)^{k}}{(z-\lambda)^{k}}\bigl(z-A_{X,Z}\bigr)^{-1}C_{X,Z}^{k+1} +\sum
\limits^{k}_{i=1}\frac{(-1)^{k-i}\bigl(\lambda-A_{X,Z}\bigr)^{-i}C_{X,Z}^{k+1}}{\bigl(z-\lambda\bigr)^{k+1-i}}.
\end{align*}

\item[(b)] Let $n\in {\mathbb N}$ and $U\in D(A_{X,Z}^{n}).$ Then
\begin{align*}
\bigl(\lambda &-A_{X,Z}\bigr)^{-1}C_{X,Z}U
={\lambda}^{-1}C_{X,Z}U+{\lambda}^{-2}C_{X,Z}A_{X,Z}U
\\&+\cdot \cdot \cdot+{\lambda}^{-n}C_{X,Z}A_{X,Z}^{n-1}U+{\lambda}^{-n}\bigl(\lambda-A_{X,Z}\bigr)^{-1}C_{X,Z}A_{X,Z}^{n}U.
\end{align*}
\end{itemize}

In  K\=omura's approach \cite{komura}, let   $X={\mathcal D}_{(-\infty ,a]},$ $Y=E,$ $Z=L({\mathcal D}_{(-\infty ,a]},E)$ for some $a>0,$ and $C=I_{E}$). A
linear operator $A$ in $E$
is the infinitesimal generator of a uniquely determined locally equicontinuous
$C_{0}$-semigroup $(T(t))_{t\geq 0}$ in $E$ iff the following holds:
\begin{itemize}
\item[(1)] $A$ is a closed linear operator with a dense domain $D(A);$
\item[(2)] for any $a > 0,$ in the space ${\mathbf D}^{\prime}_{a}(E)$ the following conditions are
satisfied:
\begin{itemize}
\item[(a)] there exists the generalized resolvent $(\lambda I_{Z}- {\mathbf A})^{-1}$ of $A;$
\item[(b)] for any fixed complex number $\lambda$, there is a continuous linear
operator $R(\lambda)$ on $E$ into itself such that for any fixed $x \in E,$ $R(\lambda)x$ is an
$E$-valued entire function in $\lambda$ satisfying that there exists $k\in {\mathbb N}$ such that for
any continuous seminorm $p$ on $E,$ there exist an integer $N = N(p) > 0$
and a number $C = C(p) > 0$ with
$p(R(\lambda) x) \leq C(1 + |\lambda|)^{N}e^{k|\Re \lambda|},$ $\lambda \in {\mathbb C},$
as well as that $R(\lambda)x$
is a representation
of $(\lambda I_{Z} - {\mathbf A})^{-1}$ and the family of operators
$$
\Biggl \{ \frac{\lambda^{n+1}}{n!}\frac{d^{n}}{d\lambda^{n}}R(\lambda) : \lambda >0,\ n\in {\mathbb N}_{0} \Biggr\}\subseteq L(E) \quad \mbox{is equicontinuous.}
$$

\end{itemize}
\end{itemize}
The proof of \cite[Theorem 3]{komura} is rather long and can be trivially modified only for the class of locally equicontinuous
$C$-regularized semigroups in SCLCS's (recall that there exist examples of integrated semigroups and $C$-regularized semigroups in Banach  spaces with not necessarily densely defined generators, so that we cannot expect the validity of (1) in this framework). On the other hand, some necessary and sufficent conditions for the generation of locally equicontinuous $K$-convoluted $C$-semigroups in SCLCS's, defined locally or globally, can be very simply clarified if we use the notion of asymptotic $\Theta C$-resolvents (cf. S. \=Ouchi \cite{ouchi} for the pioneering results in this direction, \cite{kuosi}, \cite{t212}, \cite{w18} and \cite{knjigah} for the Banach space case): Let $0<\tau \leq \infty,$ let $\gamma\in[0,\tau),$ and let $K\in L_{loc}^{1}([0,\tau)),$ $K\neq 0.$ Set $\Theta (t):=\int^{t}_{0}K(s)\, ds,$ $t\in [0,\tau).$ An operator family $\{L_{\gamma}(\lambda):\gamma\in[0,\tau)$, $\lambda\geq 0\}\subseteq L(E)$
is called an asymptotic $\Theta C$-resolvent for $A$ iff there exists a strongly continuous operator family
$(V(t))_{t\in[0,\tau)}\subseteq L(E)$ such that the following conditions hold:
\begin{itemize}
\item[(i)] For every fixed element $x\in E,$ the function $\lambda\to L_{\gamma} (\lambda)x,$ $\lambda \geq 0$
belongs to $C^{\infty}([0,\infty):E)$ and the operator family
\[
\Biggl\{ \frac{\lambda^n}{(n-1)!}
\frac{d^{n-1}}{d\lambda^{n-1}} L_{\gamma}(\lambda) : \lambda\geq 0,\;n\in\mathbb{N} \Biggr\}\subseteq L(E)
\]
is equicontinuous.
\item[(ii)] $L_{\gamma}(\lambda)$ commutes with $C$ and $A$ for all $\lambda\geq0$.
\item[(iii)] $(\lambda-A) L_{\gamma}(\lambda)x=-e^{-\lambda\gamma}V(\gamma)x
+\int^{\gamma}_0 e^{-\lambda s}K(s)Cx\,ds$, $\lambda\geq0$.
\item[(iv)] $L_{\gamma}(\lambda) L_{\gamma}(\eta)=L_{\gamma}(\eta)L_{\gamma}(\lambda)$,
$\lambda\geq 0$, $\eta\geq 0$.
\end{itemize}
Keeping this notion in mind, it can be straightforwardly verified that
the assertions of \cite[Proposition 2.3.18, Theorem 2.3.19-Theorem 2.3.20]{knjigah} continue to hold in locally convex spaces with minor technical modifications.

%

Following \cite[Definition 2.1]{ush1}, it will be said that an $L(E)$-valued distribution (ultradistribution of $\ast$-class) ${\mathcal G}$ is boundedly equicontinuous
iff for every $p\in \circledast$ and for every bounded subset $B$ of ${\mathcal D}$ (${\mathcal D}^{\ast}$), there exist $c>0$ and
$q\in \circledast$ such that
$$
p({\mathcal G}(\varphi)x)\leq cq(x),\quad \varphi \in B, \ x\in E.
$$
If $E$ is barreled, then the uniform boundedness principle \cite[p. 273]{meise} implies that each ${\mathcal G}\in {\mathcal D}'(L(E))$ (${\mathcal G}\in {\mathcal D}'^{*}(L(E))$) is automatically boundedly equicontinuous.

Suppose now that the operator  $A$ is $C$-wellposed for the abstract Cauchy problem $u'-Au=G$ at $t=0$ in the sense of distributions
(ultradistributions of $\ast$-class); for the sake of brevity, we will consder only the ultradistribution case. Put $G_{x}(\varphi):=\varphi(0)x,$ $\varphi \in {\mathcal D}^{\ast}$ ($x\in X$).
Then we define ${\mathcal G}(\varphi)x:=U_{G_{x}}(\varphi),$ $\varphi \in {\mathcal D}^{\ast}$ ($x\in X$). Using the fact that the space ${\mathcal D}^{\ast}$ is barelled and the arguments already used in the proof of \cite[Theorem 2.1]{ush1}, we can simply prove the following theorem.

\begin{thm}\label{59}
Let $A$ be $C$-wellposed, let $C\in L(E)$ be injective, and let $CA\subseteq AC$. Then there exists a boundedly equicontinous ${\mathcal G}\in\DD^{\prime}_{0}(L(E))$ $({\mathcal G}\in {\mathcal D}^{\prime \ast}_{0}(L(E)))$ satisfying the following properties:
\begin{itemize}
\item[(i)] For any $x\in E$ and $\varphi\in\DD$ $(\varphi\in\DD^{\ast})$, we have ${\mathcal G}(\varphi)x\in D(A)$ and ${\Big(}\frac{d}{dt}{\mathcal G}{\Big)}(\varphi)x-A{\mathcal G}(\varphi)x= \delta(\varphi)Cx$;
\item[(ii)] For any $x\in D(A)$, $\varphi\in\DD$ $(\varphi\in\DD^{\ast})$, we have ${\mathcal G}(\varphi)Ax=A{\mathcal G}(\varphi)x$;
\item[(iii)] For any  $\varphi\in\DD$ $(\varphi\in\DD^{\ast})$, we have ${\mathcal G}(\varphi)Cx=C{\mathcal G}(\varphi)x.$
\end{itemize}
\end{thm}
\begin{proof}
If $E$ is sequentially complete, then the opposite direction of Theorem \ref{59} holds.
In order to prove that, we need two auxiliaries lemmas whose proofs in ultradistribution case can be deduced by slightly modifying the corresponding proofs of \cite[Proposition 2.1, Proposition 2.2]{ush1}.\end{proof}

\begin{lem}\label{lemma1}
The algebraic tensor product ${\mathbf D}^{\prime \ast}_0\otimes E$ is dense in ${\mathbf D}^{\prime \ast}_0(E)$. If $A$ is closed linear operator on $E$, then ${\mathbf D}^{\prime \ast}_0\otimes D(A)$ is dense in $D({\mathbf A})$ topologized by the graph topology of ${\mathbf A}$.
\end{lem}

\begin{lem}\label{lemma2} (cf. also \cite[Section 3]{ku113} for Banach space valued ultradistributions)
For any boundedly equicontinuous ultradistribution ${\mathcal G}\in\DD^{\prime \ast}_0(L(E))$, there exists a unique convolution operator ${\mathcal G}\ast \cdot \in L(\DD^{\prime \ast}_{0}(E))$ satisfying that for $f=F\otimes x$ (defined in the obvious way), with arbitrary $F\in\DD^{\prime \ast}_0$ and $x\in E,$ we have:
$$
({\mathcal G}\ast f)(\varphi)={\mathcal G}_t\bigl(\alpha(t)F_s(\varphi(t+s))\bigr)x,
$$
where $\alpha(t)$ is an arbitrary smooth function with \emph{supp}$(\alpha)\subset[a,\infty)$, $a>-\infty$ and $\alpha(t)=1$ for $t\geq0$.
\end{lem}

\begin{thm}\label{opositethm}
Let $E$ be a sequentially complete locally convex space, let $C\in L(E)$ be an injective operator, and let $CA\subseteq AC$. Then a closed linear operator $A$ on $E$ is $C$-wellposed iff there exists a boundedly equicontinuous ${\mathcal G}\in\DD^{\prime}_0(L(E))$ (${\mathcal G}\in\DD^{\prime \ast}_0(L(E))$) satisfying \emph{(i)-(iii)} of \emph{Theorem \ref{59}}.
\end{thm}
\begin{proof}
The sufficiency can be given directly, by simple showing that for any  boundedly equicontinuous ultradistribution ${\mathcal G}\in\DD^{\prime \ast}_0(L(E))$, $G\mapsto U_{G}:={\mathcal G} \ast G$ is a unique mapping  belonging to $L( {\mathcal D}'_{0}(E))$ and satisfying the properties (i)-(iii) from Definition \ref{C-wellposed}.
\end{proof}
Observe, finally, that Theorem \ref{59} and Theorem \ref{opositethm} can be simply reformulated in the case that $A$ is exponentially $C$-wellposed.

\section{The basic properties of C-distribution semigroups and
C-ultradistribution semigroups in locally convex spaces}


\begin{defn}\label{cuds}
Let $\mathcal{G}\in\mathcal{D}_0'(L(E))$ ($\mathcal{G}\in\mathcal{D}_0'^{\ast}(L(E))$) satisfy $C\mathcal{G}=\mathcal{G}C,$
and let $\mathcal{G}$ be boundedly equicontinuous. Then it is said that $\mathcal{G}$ is a pre-(C-DS) (pre-(C-UDS) of $\ast$-class) iff the following holds:
\[\tag{C.S.1}
\mathcal{G}(\varphi*_0\psi)C=\mathcal{G}(\varphi)\mathcal{G}(\psi),\quad \varphi,\;\psi\in\mathcal{D} \ \ (\varphi,\;\psi\in\mathcal{D}^{\ast}).
\]
If, additionally,
\[\tag{C.S.2}
\mathcal{N}(\mathcal{G}):=\bigcap_{\varphi\in\mathcal{D}_0}N(\mathcal{G}(\varphi))=\{0\} \ \  \Biggl(\mathcal{N}(\mathcal{G}):=\bigcap_{\varphi\in\mathcal{D}^{\ast}_0}N(\mathcal{G}(\varphi))=\{0\}\Biggr),
\]
then $\mathcal{G}$ is called a $C$-distribution semigroup ($C$-ultradistribution semigroup of $\ast$-class), (C-DS) ((C-UDS)) in short.
A pre-(C-DS) $\mathcal{G}$ is called dense if
\[\tag{C.S.3}
\mathcal{R}(\mathcal{G}):=\bigcup\limits_{\varphi\in\mathcal{D}_0}R(\mathcal{G}(\varphi))
\text{ is dense in }E\
\Biggl(\mathcal{R}(\mathcal{G}):=\bigcup\limits_{\varphi\in\mathcal{D}^{\ast}_0}R(\mathcal{G}(\varphi))
\text{ is dense in }E \Biggr).
\
\]
The notion of a dense pre-(C-UDS) $\mathcal{G}$ of $\ast$-class (and the set $\mathcal{R}(\mathcal{G})$) is defined similarly.
\end{defn}

\begin{rem}\label{redun}
\begin{itemize}
\item[(i)]
We have assumed that ${\mathcal G}$ is boundedly equicontinuous in order to stay consistent with the notion introduced in \cite{ush1} and our previous analysis. Observe, however, that the assumption on bounded equicontinuity of ${\mathcal G}$
is slightly redundant and that we can rephrased a great part of our results in the case that ${\mathcal G}$ does not satisfy this condition.
\item[(ii)] If $C=I,$ then we also write pre-(DS), pre-(UDS), (DS), (UDS), ... , instead of pre-(C-DS), pre-(C-UDS), (C-DS), (C-UDS).
\end{itemize}
\end{rem}

Suppose that $\mathcal{G}$ is a pre-(C-DS) (pre-(C-UDS) of $\ast$-class). Then
$\mathcal{G}(\varphi)\mathcal{G}(\psi)=\mathcal{G}(\psi)\mathcal{G}(\varphi)$ for all $\varphi,\,\psi\in\mathcal{D}$ ($\varphi,\,\psi\in\mathcal{D}^{\ast}$),
and $\mathcal{N}(\mathcal{G})$ is a closed subspace of $E$.

The structural characterization of a pre-(C-DS) $\mathcal{G}$ (pre-(C-UDS) $\mathcal{G}$ of $\ast$-class) on its kernel space
$\mathcal{N}(\mathcal{G})$ is described in the following theorem (cf. Theorem \ref{delta}, the paragraph directly after its formulation, as well as \cite[Proposition 3.1.1]{knjigah} and the proofs of \cite[Lemma 2.2]{ku112}, \cite[Proposition 3.5.4]{knjigah}).

\begin{thm}\label{delta-point}
\begin{itemize}
\item[(i)] Let $\mathcal{G}$ be a pre-$($C-DS$)$, and let the space $L(\mathcal{N}(\mathcal{G}))$ be admissible.
Then, with $N=\mathcal{N}(\mathcal{G})$ and $G_1$ being the restriction of $\mathcal{G}$ to $N$ $(G_1=\mathcal{G}_{|N})$,
we have:
There exists a unique operators $T_0$, $T_1,\dots,T_m\in L(\mathcal{N}(\mathcal{G}))$ such that
$G_1=\sum_{j=1}^m\delta^{(j)}\otimes T_j$, $T_iC^i=(-1)^iT_0^{i+1}$, $i=0,1,\dots,m-1$ and
$T_0T_m=T_0^{m+2}=0$.
\item[(ii)]  Let $(M_{p})$ satisfy \emph{(M.3)}, let $\mathcal{G}$ be a pre-$($C-UDS$)$ of $\ast$-class, and let the space $\mathcal{N}(\mathcal{G})$ be barreled.
Then, with $N=\mathcal{N}(\mathcal{G})$ and $G_1$ being the restriction of $\mathcal{G}$ to $N$ $(G_1=\mathcal{G}_{|N})$,
we have:
There exists a unique set of operators $(T_{j})_{j\in {\mathbb N}_{0}}$ in $L(\mathcal{N}(\mathcal{G}))$ such that
$G_1=\sum_{j=0}^{\infty}\delta^{(j)}\otimes T_j$, $T_jC^j=(-1)^jT_0^{j+1}$, $j\in {\mathbb N}$ and the set $\{M_{j}T_{j}L^{j} : j\in {{\mathbb N}_{0}}\}$ is bounded in
$L(\mathcal{N}(\mathcal{G})),$ for some $L>0$ in the Beurling case, resp. for every $L>0$ in the Roumieu case.
\end{itemize}
\end{thm}

Let $\mathcal{G}\in\mathcal{D}_0'(L(E))$ ($\mathcal{G}\in\mathcal{D}_0'^{\ast}(L(E))$) satisfy (C.S.2), and let $T\in\mathcal{E}_0'$ ($T\in\mathcal{E}_0'^{\ast}$),
i.e., $T$ is a scalar-valued distribution (ultradistribution of $\ast$-class) with compact support contained in $[0,\infty)$.
Define
\[
G(T)x:=\Bigl\{(x,y) \in E\times E : \mathcal{G}(T*\varphi)x=\mathcal{G}(\varphi)y\;\mbox{ for all }\;\varphi\in\mathcal{D}_0 \ \ \bigl(\varphi\in\mathcal{D}^{\ast}_{0}\bigr) \Bigr\}.
\]
Then  it can be easily seen that $G(T)$ is a closed linear operator.
Following R. Shiraishi, Y. Hirata \cite{1964} and P. C. Kunstmann \cite{ku112}, we define the (infinitesimal) generator of a (C-DS) $\mathcal{G}$ by $A:=G(-\delta')$ (for some other approaches, see J. L. Lions \cite{li121}, \cite[Remark 3.1.20]{knjigah} and J. Peetre \cite{peet}, T. Ushijima \cite{ush1}).
Since for each $\psi\in\mathcal{D}$ ($\psi\in\mathcal{D}^{\ast}$), we have $\psi_+:=\psi\mathbf{1}_{[0,\infty)}\in\mathcal{E}_0'$ ($\mathcal{E}_{0}'^{*}$),
($\mathbf{1}_{[0,\infty)}$ stands for the characteristic function of $[0,\infty)$) the definition of $G(\psi_+)$ is clear.
Further on, if $\mathcal{G}$ is a (C-DS) ((C-UDS) of $\ast$-class), $T\in\mathcal{E}_0'$ ($T\in\mathcal{E}_{0}'^{*}$) and $\varphi\in\mathcal{D}$ ($\varphi\in\mathcal{D}^{\ast}$),
then ${\mathcal G}(\varphi)G(T)\subseteq G(T)\mathcal{G}(\varphi)$, $CG(T)\subseteq G(T)C$
and $\mathcal{R}(\mathcal{G})\subseteq D(G(T))$.
If $\mathcal{G}$ is a pre-(C-DS) (pre-(C-UDS) of $\ast$-class) and $\varphi$, $\psi\in\mathcal{D}$ ($\varphi$, $\psi\in\mathcal{D}^{\ast}$),
then the assumption $\varphi(t)=\psi(t)$, $t\geq 0$, implies $\mathcal{G}(\varphi)=\mathcal{G}(\psi)$.
As in the Banach space case, we can prove the following: Suppose that $\mathcal{G}$ is a (C-DS) ((C-UDS) of $\ast$-class). Then $G(\psi_+)C=\mathcal{G}(\psi)$, $\psi\in\mathcal{D}$ ($\psi\in\mathcal{D}^{\ast}$) and $C^{-1}AC=A.$ Furthermore, the following holds:

\begin{prop}\label{isto}
Let ${\mathcal G}$ be a (C-DS) ((C-UDS) of $\ast$-class), $S$, $T\in\mathcal{E}'_0$ ($S$, $T\in\mathcal{E}'^{\ast}_0$), $\varphi\in\mathcal{D}_0$ ($\varphi\in\mathcal{D}^{\ast}_0$), $\psi\in\mathcal{D}$
($\psi\in\mathcal{D}^{\ast}$) and $x\in E$.
Then we have:
\begin{itemize}
\item[(i)] $(\mathcal{G}(\varphi)x$, $\mathcal{G}(\overbrace{T*\cdots*T}^m*\varphi)x)\in G(T)^m$, $m\in\mathbb{N}$.
\item[(ii)] $G(S)G(T)\subseteq G(S*T)$ with $D(G(S)G(T))=D(G(S*T))\cap D(G(T))$, and $G(S)+G(T)\subseteq G(S+T)$.
\item[(iii)] $(\mathcal{G}(\psi)x$, $\mathcal{G}(-\psi')x-\psi(0)Cx)\in G(-\delta')$.
\item[(iv)] If $\mathcal{G}$ is dense, then its generator is densely defined.
\end{itemize}
\end{prop}

The assertions (ii)-(vi) of \cite[Proposition 3.1.2]{knjigah} can be reformulated for pre-(C-DS)'s (pre-(C-UDS)'s of $\ast$-class) in locally convex spaces; here it is only worth noting that for any barreled space $E$ and for any bounded subset $B$ of $E^*$ the mapping
$x\mapsto \sup_{x^{*}\in B}|\langle x^{\ast},x\rangle |,$ $x\in E$ is a continuous seminorm on $E$ (cf. also the proof of \cite[Theorem 2.3]{ush1}) and that the reflexivity of state space $E$ (recall that the sequential completeness of $E$ is our standing hypothesis)  implies that the spaces $E,$ $E^*$ and $E^{**}=E$ are both barreled and sequentially complete.

\begin{prop}\label{kuki}
Let $\mathcal{G}$ be a pre-(C-DS) (pre-(C-UDS) of $\ast$-class). 
 Then the following holds:
\begin{itemize}
\item[(i)] $C(\overline{\langle\mathcal{R}(\mathcal{G})\rangle})\subseteq\overline{\mathcal{R}(\mathcal{G})}$,
where $\langle\mathcal{R}(\mathcal{G})\rangle$
denotes the linear span of $\mathcal{R}(\mathcal{G})$.
\item[(ii)] Assume $\mathcal{G}$ is not dense and
$\overline{C\mathcal{R}(\mathcal{G})}=\overline{\mathcal{R}(\mathcal{G})}$.
Put $R:=\overline{\mathcal{R}(\mathcal{G})}$ and $H:=\mathcal{G}_{|R}$.
Then $H$ is a dense pre-($C_1$-DS)  (pre-($C_1$-UDS) of $\ast$-class) on $R$ with $C_1=C_{|R}$.
\item[(iii)] Assume $\overline{R(C)}=E$ and $E$ is barreled.
Then the dual $\mathcal{G}(\cdot)^*$ is a pre-($C^*$-DS) (pre-($C^*$-UDS) of $\ast$-class) on $E^*$
and $\mathcal{N}(\mathcal{G}^*)=\overline{\mathcal{R}(\mathcal{G})}^{\circ}$.
\item[(iv)] If $E$ is reflexive and $\overline{R(C)}=E$,
then $\mathcal{N}(\mathcal{G})=\overline{\mathcal{R}(\mathcal{G}^*)}^{\circ}$.
\item[(v)] Assume $\overline{R(C)}=E$ and $E$ is barreled.
Then $\mathcal{G}^*$ is a ($C^*$-DS) (($C^*$-UDS) of $\ast$-class) in $E^*$ iff $\mathcal{G}$ is a dense pre-(C-DS) (pre-(C-UDS) of $\ast$-class).
If $E$ is reflexive, then $\mathcal{G}^*$ is a dense pre-($C^*$-DS) (pre-($C^*$-UDS) of $\ast$-class) in $E^*$ iff $\mathcal{G}$ is a (C-DS) ((C-UDS) of $\ast$-class).
\end{itemize}
\end{prop}


Now we shall state and prove an extension of \cite[Proposition 2]{ki90} for pre-(C-DS)'s (pre-(C-UDS)'s of $\ast$-class) in locally convex spaces.

\begin{prop}\label{kisinski}
Suppose that ${\mathcal G}\in {\mathcal D}^{\prime}_{0}(L(E))$ (${\mathcal G}\in {\mathcal D}^{\prime \ast}_{0}(L(E))$) and ${\mathcal G}(\varphi)C=C{\mathcal G}(\varphi),$ $\varphi \in {\mathcal D}$ ($\varphi \in {\mathcal D}^{\ast}$).
Then ${\mathcal G}$ satisfies \emph{(C.S.1)} iff
\begin{equation}\label{polish}
{\mathcal G}\bigl(\varphi^{\prime}\bigr){\mathcal G}(\psi)-{\mathcal G}(\varphi){\mathcal G}\bigl(\psi^{\prime}\bigr)=\psi(0){\mathcal G}(\varphi)C-\varphi(0){\mathcal G}
(\psi)C,\quad \varphi,\ \psi \in {\mathcal D} \ \ \bigl(  \varphi,\ \psi \in {\mathcal D}^{\ast} \bigr).
\end{equation}
In particular, ${\mathcal G}$ is a pre-(C-DS) (pre-(C-UDS) of $\ast$-class) iff ${\mathcal G}$ is boundedly equicontinuous and \emph{(\ref{polish})} holds.
\end{prop}

\begin{proof}
Steps of the proof of the proposition are the same as that of \cite[Proposition 2]{ki90}; because of its significance, we shall include all relevant details of the proof.
If ${\mathcal G}$ satisfies (C.S.1), then (\ref{polish}) follows immediately from (C.S.1) and the equality $\varphi^{\prime}\ast_{0} \psi -\varphi \ast_{0} \psi^{\prime}=\psi(0)\varphi -\varphi(0)\psi ,$ $\varphi,\ \psi \in {\mathcal D}$ ($\varphi,\ \psi \in {\mathcal D}^{\ast}$).
Suppose now that (\ref{polish}) holds, $\varphi,\ \psi \in {\mathcal D}$ ($\varphi,\ \psi \in {\mathcal D}^{\ast}$), $a>0$ and supp$(\psi)\subseteq (-\infty,a].$
Since ${\mathcal G}\in {\mathcal D}^{\prime}_{0}(L(E))$ (${\mathcal G}\in {\mathcal D}^{\prime \ast}_{0}(L(E))$), and the function $t\mapsto \int^{a}_{0}[\varphi(t-s)\psi(s)-\varphi(-s)\psi(t+s)]\, ds,$ $t\in {\mathbb R}$ belongs to ${\mathcal D}$ (${\mathcal D}^{\ast}$) with
$(\varphi \ast_{0}\psi)(t)=\int^{a}_{0}[\varphi(t-s)\psi(s)-\varphi(-s)\psi(t+s)]\, ds,$ $t\geq 0,$ we have
\begin{align}
\notag {\mathcal G}(\varphi \ast_{0}\psi)Cx&={\mathcal G}\int^{a}_{0}\bigl[\varphi(\cdot-s)\psi(s)-\varphi(-s)\psi(\cdot+s)\bigr]Cx\, ds
\\\notag &=\int^{a}_{0}\bigl[\psi(s){\mathcal G}(\varphi(\cdot-s))Cx-\varphi(-s){\mathcal G}(\psi(\cdot+s))Cx\bigr]\, ds
\\\label{polish-1} &=\int^{a}_{0}\Bigl[{\mathcal G}\bigl(\varphi^{\prime}(\cdot-s)\bigr)
{\mathcal G}(\psi(\cdot +s))x-
{\mathcal G}(\varphi (\cdot-s))
{\mathcal G}\bigl(\psi^{\prime}(\cdot +s)\bigr)x\Bigr]\, ds
\\\label{polish-2} & =-\int^{a}_{0}\frac{d}{ds}\bigl[{\mathcal G}(\varphi(\cdot-s)){\mathcal G}(\psi(\cdot +s))x\bigr]\, ds
\\\notag &={\mathcal G}(\varphi) {\mathcal G}(\psi)x-{\mathcal G}(\varphi(\cdot -a)){\mathcal G}(\psi (\cdot+a))x
\\\notag &={\mathcal G}(\varphi) {\mathcal G}(\psi)x-{\mathcal G}(\varphi(\cdot -a))0x={\mathcal G}(\varphi) {\mathcal G}(\psi)x,
\end{align}
for any $x\in E$ and $\varphi,\ \psi \in {\mathcal D}$ ($\varphi,\ \psi \in {\mathcal D}^{\ast}$),
where (\ref{polish-1}) follows from an application of (\ref{polish}), and  (\ref{polish-2}) from an elementary argumentation involving the continuity of ${\mathcal G}$ as well as the facts that for each function $\zeta \in {\mathcal D}$ ($\zeta \in {\mathcal D}^{\ast}$)
we have that $\lim_{h\rightarrow 0}(\tau_{h}\zeta)=\zeta$ in ${\mathcal D}$ (${\mathcal D}^{\ast}$), $\lim_{h\rightarrow 0}\frac{1}{h}(\tau_{h}\zeta -\zeta)=\zeta^{\prime}$ in ${\mathcal D}$ (${\mathcal D}^{\ast}$)
and that the set $\{\tau_{h}\zeta : |h|\leq 1\}$ is bounded in ${\mathcal D}$ (${\mathcal D}^{\ast}$).
The proof of proposition is thereby complete.
\end{proof}

\begin{thm}\label{fundamentalna}
Suppose that ${\mathcal G}\in {\mathcal D}^{\prime}_{0}(L(E))$ (${\mathcal G}\in {\mathcal D}^{\prime \ast}_{0}(L(E))$), ${\mathcal G}(\varphi)C=C{\mathcal G}(\varphi),$ $\varphi \in {\mathcal D}$ ($\varphi \in {\mathcal D}^{\ast}$)
and $A$ is a closed linear operator on $E$ satisfying that $\mathcal{G}(\varphi)A\subseteq A{\mathcal G}(\varphi),$ $\varphi \in {\mathcal D}$ ($\varphi \in {\mathcal D}^{\ast}$) and
\begin{equation}\label{dkenk}
A\mathcal{G}(\varphi)x=\mathcal{G}\bigl(-\varphi'\bigr)x-\varphi(0)Cx,\quad x\in E,\ \varphi \in {\mathcal D} \ \ (\varphi \in {\mathcal D}^{\ast}).
\end{equation}
Then the following holds:
\begin{itemize}
\item[(i)] ${\mathcal G}$ satisfies \emph{(C.S.1)}.
\item[(ii)] If ${\mathcal G}$ is boundedly equicontinuous and satisfies \emph{(C.S.2)}, then ${\mathcal G}$ is a (C-DS) ((C-UDS) of $\ast$-class) generated by $C^{-1}AC.$
\item[(iii)] Consider the distribution case. If $E$ is admissible, then the condition \emph{(C.S.2)} automatically holds for ${\mathcal G}$.
\end{itemize}
\end{thm}

\begin{proof}
Let $\varphi,\ \psi \in {\mathcal D}$ ($\varphi,\ \psi \in {\mathcal D}^{\ast}$) be fixed. Using the inclusion $\mathcal{G}(\varphi)A\subseteq A{\mathcal G}(\varphi)$ and the equality (\ref{dkenk}), we get
that
$
A{\mathcal G}(\varphi){\mathcal G}(\psi)x={\mathcal G}(\varphi)A{\mathcal G}(\psi)x,
$ $x\in E,$ i.e.,
\begin{equation}\label{nahari}
{\mathcal G}\bigl( -\varphi^{\prime}\bigr){\mathcal G}(\psi)x-\varphi(0)C{\mathcal G}(\psi)x={\mathcal G}(\varphi)\Bigl[ {\mathcal G}\bigl( -\psi^{\prime} \bigr)x-\psi(0)Cx \Bigr],\quad x\in E.
\end{equation}
This is, in fact, (\ref{polish}) so that (i) follows immediately from Proposition \ref{kisinski}. The proof of (iii) is the same as in the Banach space case (see e.g. \cite{li121} and the proof of \cite[Theorem 3.1.27]{knjigah}). Hence, it remains to be proved that
the integral generator of ${\mathcal G}$ is the operator $C^{-1}AC,$ if ${\mathcal G}$ is boundedly equicontinuous and satisfies (C.S.2) (cf. the item (ii)); for the sake of brevity, we shall consider only the
distribution case. Denote by $B$ the integral generator of ${\mathcal G}.$ Then it is checked at once that
$C^{-1}AC\subseteq B.$ Suppose now that $(x,y)\in B,$ i.e., that ${\mathcal G}(-\zeta^{\prime})x={\mathcal G}(\zeta)y$ for all $\zeta \in {\mathcal D}_{0}.$ This clearly implies $A{\mathcal G}(\zeta)x={\mathcal G}(\zeta)y$ for all $\zeta \in {\mathcal D}_{0}.$
Consider now the equation (\ref{nahari}) with $\varphi=\xi$ and $\xi(0)=1.$ By (\ref{dkenk}), it readily follows that $C{\mathcal G}(\psi)x\in D(A).$ Since $A{\mathcal G}(\zeta)x={\mathcal G}(\zeta)y$ for all $\zeta \in {\mathcal D}_{0},$
we obtain that ${\mathcal G}(\zeta)AC{\mathcal G}(\eta)x={\mathcal G}(\zeta)C{\mathcal G}(\eta)y$ for all $\zeta \in {\mathcal D}_{0}$ ($\eta \in {\mathcal D}$). By (C.S.2), we get that
$AC{\mathcal G}(\eta)x=C{\mathcal G}(\eta)y$  ($\eta \in {\mathcal D}$). This, in turn, implies $CA{\mathcal G}(\eta)x=C{\mathcal G}(\eta)y,$ $A{\mathcal G}(\eta)x={\mathcal G}(\eta)y,$
${\mathcal G}( -\eta^{\prime})x-\eta(0)Cx={\mathcal G}(\eta)y$ ($\eta \in {\mathcal D}$), and since $\eta$ was arbitrary, $Cx\in D(A).$ Combined with the equality $AC{\mathcal G}(\eta)x=C{\mathcal G}(\eta)y$  ($\eta \in {\mathcal D}$), the above implies $ACx=Cy$ and $C^{-1}ACx=y,$ as claimed.
\end{proof}

\begin{rem}
\begin{itemize}
\item[(i)]
Even in the case that $E$ is a Banach space and $C=I,$ ${\mathcal G}$ need not satisfy the condition (C.S.2) in ultradistribution case (\cite{cizi}, \cite{knjigah}).
\item[(ii)]
Suppose that $\overline{R(C)}=E,$ $E$ is barreled and $A$ generates a dense (C-DS) ((C-UDS) of $\ast$-class) on $E$.
Then Proposition \ref{kuki}(iii) implies that the dual $\mathcal{G}(\cdot)^*$ is a ($C^*$-DS) (($C^*$-UDS) of $\ast$-class) on $E^*.$ Since ${\mathcal G}^{\ast}(\varphi)A^{\ast} \subseteq A^{\ast}{\mathcal G}^{\ast}(\varphi),$ $\varphi \in {\mathcal D}$ ($\varphi \in {\mathcal D}^{\ast}$) and (\ref{dkenk}) holds with $A$ and ${\mathcal G}$ replaced respectively by $A^{\ast}$ and ${\mathcal G}^{\ast},$ Theorem \ref{fundamentalna}(ii) implies that the integral generator of ${\mathcal G}^{\ast}$ is the operator $(C^{\ast})^{-1}A^{\ast}C^{\ast}$ (cf. also \cite[Remark 3.1.22]{knjigah}).
\item[(iii)] Using Proposition \ref{kisinski} and the method proposed by J. Kisy\'nski in \cite[Section 6]{ki90}, we can introduce the
integral generator of a pre-(C-DS) (pre-(C-UDS) of $\ast$-class) on an arbitrary sequentially complete locally convex space. On the other hand, it seems that the method proposed in \cite[Section 3.5; cf. Definition 3.5.7]{knjigah}
can be used to define the integral generator of a pre-(C-DS) (pre-(C-UDS) of $\ast$-class) only in the case that $E$ is a Banach space.
It would take too long to go into further details concerning these subjects here.
\end{itemize}
\end{rem}

\begin{prop}\label{kisinski-uniqueness}
Every $C$-distribution semigroup ($C$-ultradistribution semigroup of $\ast$-class) is uniquely determined by its generator.
\end{prop}

\begin{proof}
We will prove the assertion only in distribution case because the ultradistribution case can be considered quite similarly. Suppose that ${\mathcal G}_{1}$ and ${\mathcal G}_{2}$ are two $C$-distribution semigroups generated by
$A.$ Put ${\mathcal G}:={\mathcal G}_{1}-{\mathcal G}_{2}.$ Then ${\mathcal G} \in {\mathcal D}^{\prime}_{0}(L(E)),$ $A{\mathcal G}(\varphi)\subseteq {\mathcal G}(\varphi)A,$ $\varphi \in {\mathcal D}$ and
$A{\mathcal G}(\varphi)={\mathcal G}(-\varphi^{\prime}),$ $\varphi \in {\mathcal D}.$
This implies
\begin{align*}
{\mathcal G}\bigl(\varphi^{\prime}\bigr){\mathcal G}(\psi)x=-A{\mathcal G}(\varphi){\mathcal G}(\psi)x=-{\mathcal G}(\varphi)A{\mathcal G}(\psi)x={\mathcal G}(\varphi){\mathcal G}\bigl(\psi^{\prime}\bigr)x,
\end{align*}
for all $\varphi,\ \psi \in {\mathcal D}$ and $x\in E.$ Keeping in mind this equality,
the part of proof of Proposition \ref{kisinski} starting from the equation (\ref{polish-1}) shows that
\begin{align*}
0&=\int^{a}_{0}\Bigl[{\mathcal G}\bigl(\varphi^{\prime}(\cdot-s)\bigr)
{\mathcal G}(\psi(\cdot +s))x-
{\mathcal G}(\varphi (\cdot-s))
{\mathcal G}\bigl(\psi^{\prime}(\cdot +s)\bigr)x\Bigr]\, ds={\mathcal G}(\varphi){\mathcal G}(\psi)x
\end{align*}
for all $\varphi,\ \psi \in {\mathcal D}$ and $x\in E$ (here, the number $a>0$ is chosen so that supp$(\psi)\subseteq (-\infty,a]$).  In particular, ${\mathcal G}(\psi){\mathcal G}(\varphi)={\mathcal G}(\varphi){\mathcal G}(\psi)=0$ ($\varphi,\ \psi \in {\mathcal D}$), so that
\begin{equation}\label{acid-prsute}
{\mathcal G}_{1}(\varphi){\mathcal G}_{2}(\psi)+{\mathcal G}_{2}(\varphi){\mathcal G}_{1}(\psi)={\mathcal G}_{1}(\psi){\mathcal G}_{2}(\varphi)+{\mathcal G}_{2}(\psi){\mathcal G}_{1}(\varphi),\quad \varphi,\ \psi \in {\mathcal D}.
\end{equation}
Applying the operator $A$ on the both sides of (\ref{acid-prsute}), and using (\ref{acid-prsute}) once more for the equality of terms ${\mathcal G}_{1}(-\varphi^{\prime}){\mathcal G}_{2}(\psi)+{\mathcal G}_{2}(-\varphi^{\prime}){\mathcal G}_{1}(\psi)$
and ${\mathcal G}_{1}(\psi){\mathcal G}_{2}(-\varphi^{\prime})+{\mathcal G}_{2}(\psi){\mathcal G}_{1}(-\varphi^{\prime}),$
we get that
\begin{align}
\notag & {\mathcal G}_{1}(\psi)  {\mathcal G}_{2}\bigl(-\varphi^{\prime}\bigr)-\varphi(0){\mathcal G}_{2}(\psi)C+{\mathcal G}_{2}(\psi){\mathcal G}_{1}\bigl(-\varphi^{\prime}\bigr)-\varphi(0){\mathcal G}_{1}(\psi)C
\\\label{funk} & ={\mathcal G}_{1}\bigl( -\psi^{\prime} \bigr){\mathcal G}_{2}(\varphi)-\psi(0){\mathcal G}_{2}(\varphi)C+{\mathcal G}_{2}\bigl(-\psi^{\prime}\bigr){\mathcal G}_{1}(\varphi)-\psi(0){\mathcal G}_{1}(\varphi)C,\quad \varphi,\ \psi \in {\mathcal D}.
\end{align}
Now we shall apply the operator $A$ on the both sides of (\ref{funk}). In such a way, we get that
\begin{align*}
{\mathcal G}_{1}\bigl(-\psi^{\prime}\bigr) &{\mathcal G}_{2}\bigl(-\varphi^{\prime}\bigr)-\psi(0){\mathcal G}_{2}\bigl( -\varphi^{\prime}\bigr)-\varphi(0) \bigl[ {\mathcal G}_{2}\bigl(-\psi^{\prime}\bigr) -\psi(0)C^{2} \bigr]
\\ & +{\mathcal G}_{2}\bigl(-\psi^{\prime}\bigr){\mathcal G}_{1}\bigl(-\varphi^{\prime}\bigr)-\psi(0){\mathcal G}_{1}\bigl( -\varphi^{\prime}\bigr)-\psi(0)\bigl[ {\mathcal G}_{1}\bigl(-\psi^{\prime}\bigr) -\psi(0)C^{2}  \bigr]
\\ & = {\mathcal G}_{1}\bigl(-\psi^{\prime} \bigr) \bigl[{\mathcal G}_{2}\bigl(-\varphi^{\prime}\bigr)-\varphi(0)C\bigr] -\psi(0)\bigl[ {\mathcal G}_{2}\bigl(-\varphi^{\prime}\bigr) -\varphi(0)C^{2} \bigr]
\\&+{\mathcal G}_{2}\bigl(-\psi^{\prime}\bigr) \bigl[ {\mathcal G}_{1}\bigl(-\varphi^{\prime}\bigr) -\varphi(0)C \bigr]-\psi(0)\bigl[ {\mathcal G}_{1}\bigl(-\varphi^{\prime}\bigr)C-\varphi(0)C^{2}  \bigr]
\end{align*}
for any $ \varphi,\ \psi \in {\mathcal D}.$ Using this equality with $\varphi(0)=1,$ and the injectivity of $C$, we obtain that ${\mathcal G}_{1}(\psi^{\prime})={\mathcal G}_{2}(\psi^{\prime}),$ $\psi \in {\mathcal D}.$ Hence, ${\mathcal G}^{\prime}=0$ and the standard arguments from the theory of scalar-valued distributions show that there exists a test function $\eta \in {\mathcal D}_{[0,1]}$ such that $\int^{+\infty}_{-\infty}\eta (t)\, dt=1$ and
${{\mathcal G}(\psi)=\mathcal G}_{1}(\psi)-{\mathcal G}_{2}(\psi)=(\int^{+\infty}_{-\infty}\psi (t)\, dt ){\mathcal G}(\eta)$ for all $\psi \in {\mathcal D}.$ Choosing $\psi \in {\mathcal D}_{[-2,-1]}$ with $\int^{+\infty}_{-\infty}\psi (t)\, dt=1$ we easily get
that ${\mathcal G}(\eta)=0,$ so ${\mathcal G}(\psi)=0$ for all $\psi \in {\mathcal D}.$ This completes the proof of proposition.
\end{proof}

\begin{rem}\label{vuaq-1977}
It should be noticed that M. Ju Vuvunikjan has observed (without giving a corresponding proof) that for any closed linear operator $A$ on $E$ there exists at most one vector-valued distribution
${\mathcal G}\in {\mathcal D}^{\prime}_{0}(L(E))$
satisfying that $\mathcal{G}(\varphi)A\subseteq A{\mathcal G}(\varphi),$ $\varphi \in {\mathcal D}$ and that (\ref{dkenk}) holds with $C=I$ (cf. \cite[p. 436]{vuaq} for more details). It is also worth noting that we do not use the operation of convolution of vector-valued (ultra-)distributions in the proof of Proposition \ref{kisinski-uniqueness}.
\end{rem}

\begin{thm}\label{lokal-int-C}
Let $\mathcal{G}$ be a (C-DS) generated by $A$, and let $L(E,[D(A)])$ be a quasi-complete \emph{(DF)}-space.
Then, for every $\tau>0$, there exist $n_{\tau}\in\mathbb{N}$
such that $A$ is the integral generator of a local $n_{\tau}$-times integrated $C$-semigroup on $E.$
\end{thm}
\begin{proof}
The proof follows by the use of \cite[Theorem 3.1.7]{knjigah}, Lemma \ref{polinomi} and the structural theorem for locally convex valued distributions.
\end{proof}

\begin{rem}\label{finite-order}
Let $\mathcal{G}$ be a (C-DS) generated by $A.$ Then we have $A\mathcal{G}(\varphi)x=-\mathcal{G}(\varphi')x-\varphi (0)Cx$, $\varphi\in\mathcal{D}$, $x\in E,$ so that
$\mathcal{G}$ can be viewed as a continuous linear mapping from $\mathcal{D}$ into $L(E,[D(A)])$. Theorem \ref{lokal-int-C} continues to hold if we assume that the distribution $\mathcal{G}\in {\mathcal D}'(L(E,[D(A)]))$ is of finite order, instead of setting the assumption that $L(E,[D(A)])$ is a quasi-complete (DF) space. In order to transfer the assertions of \cite[Theorem 3.1.21, Remark 3.1.22]{knjigah} to $C$-distribution semigroups in locally convex spaces, it seems almost inevitable to assume that the distribution $\mathcal{G}\in {\mathcal D}'(L(E,[D(A)]))$ is of finite order.
\end{rem}

The proof of subsequent theorem can be deduced by using Lemma \ref{polinomi}, the proof of
\cite[Theorem 3.1.8]{knjigah} and an elementary argumentation regarding the topological properties of the space ${\mathcal D}.$

\begin{thm}\label{cyne}
Suppose that there exists a sequence $((p_k,\tau_k))_{k\in \mathbb{N}_{0}}$ in $\mathbb{N}_{0} \times (0,\infty)$
such that $\lim_{k\rightarrow \infty}\tau_{k}=\infty$ and $A$ is a subgenerator of a locally equicontinuous $p_{k}$-times integrated
$C$-semigroup $(S_{p_{k}}(t))_{t\in [0,\tau_{k})}$ on $E$ ($k\in \mathbb{N}_{0}$).
Then the operator $C^{-1}AC$ generates a (C-DS) ${\mathcal G},$ given by
$$
{\mathcal G}(\varphi)x=(-1)^{p_{k}}\int \limits^{\infty}_{0}\varphi^{(p_{k})}(t)S_{p_{k}}(t)x\, dt,\quad \varphi \in {\mathcal D}_{(-\infty ,\tau_{k})},\ x\in E.
$$
\end{thm}

\begin{rem}\label{Banach}
In the case that $C=I,$ then it suffices to suppose that the operator $A$
is the integral generator of a locally equicontinuous $p$-times integrated
semigroup $(S_{p}(t))_{t\in [0,\tau)}$ for some $p\in {\mathbb N}$ and $\tau>0$  (cf.  \cite[Remark 3.1.10]{knjigah} for the Banach space case).
\end{rem}

Let $\alpha\in(0,\infty)\setminus \mathbb{N}$, $f\in\mathcal{S}$ and $n=\lceil\alpha\rceil$. Let us recall
that the Weyl fractional derivative $W^{\alpha}_+$ of order $\alpha$ (cf. \cite{mija} and \cite{knjigah})
is defined by
\begin{align*}
W^{\alpha}_+f(t):=\frac{(-1)^n}{\Gamma(n-\alpha)}\frac{d^n}{dt^n}\int\limits^{\infty}_t(s-t)^{n-\alpha-1}f(s)\,ds,
\;t\in\mathbb{R}.
\end{align*}
If $\alpha=n\in\mathbb{N}_{0}$, then we set $W^n_+:=(-1)^n\frac{d^n}{dt^n}.$
It is well known that the following equality holds: $W^{\alpha+\beta}_{+}f=W^{\alpha}_{+}W^{\beta}_{+}f$, $\alpha,\,\beta>0,$ $f\in\mathcal{S}$.
Suppose that  $\alpha\in(0,\infty)\setminus \mathbb{N}$ and $f\in C([0,\infty) : E).$ Set $f_{n-\alpha}(t):=(g_{n-\alpha}\ast f)(t),$ $t\geq 0.$
Making use of the dominated convergence theorem, and the change of variables $s\mapsto s-t,$ we get that
$$
\frac{1}{\Gamma(n-\alpha)}\frac{d^n}{dt^n}\int\limits^{\infty}_t(s-t)^{n-\alpha-1}\varphi(s)\,ds=
\int \limits^{\infty}_{0}g_{n-\alpha}(s)\varphi^{(n)}(t+s)\, ds,\quad t\geq 0,\ \varphi \in {\mathcal D}.
$$
Hence,
\begin{align*}
\int \limits^{\infty}_{0}W^{\alpha}_+\varphi(t)f(t)\, dt&=(-1)^{n}
\int \limits^{\infty}_{0}g_{n-\alpha}(s)\varphi^{(n)}(t+s)f(t)\, ds \, dt
\\ &=(-1)^{n}  \int \limits^{\infty}_{0}\! \!  \int \limits^{t}_{0}\varphi^{(n)}(t)g_{n-\alpha}(s)f(t-s)\, ds\, dt
\\ &=(-1)^{n}  \int \limits^{\infty}_{0}\varphi^{(n)}(t)f_{n-\alpha}(t)\, dt,\quad \varphi \in {\mathcal D}.
\end{align*}
Therefore, if
$A$ is the integral generator of a locally equicontinuous $\alpha$-times integrated $C$-semigroup $(S_{\alpha}(t))_{t\geq 0}$ on $E,$
then we have that:
$$
\int^{\infty}_0W^{\alpha}_+\varphi(t)S_{\alpha}(t)x\,dt=
(-1)^{n}  \int \limits^{\infty}_{0}\varphi^{(n)}(t)S_{n}(t)x\, dt,\quad x\in E,\ \varphi \in {\mathcal D},
$$
with $(S_{n}(t))_{t\geq 0}$ being the locally equicontinuous $n$-times integrated $C$-semigroup generated by $A.$
Combined with Theorem \ref{cyne}, the above implies:

\begin{thm}\label{miana}
Assume that $\alpha \geq 0$ and $A$ is the integral generator of a locally equicontinuous $\alpha$-times integrated $C$-semigroup $(S_{\alpha}(t))_{t\geq 0}$ on $E.$
Set
\begin{equation}\label{globalne}
\mathcal{G}_{\alpha}(\varphi)x:=\int^{\infty}_0W^{\alpha}_+\varphi(t)S_{\alpha}(t)x\,dt,\quad x\in E,\ \varphi\in\mathcal{D}.
\end{equation}
Then $A$ is the integral generator of a (C-DS) ${\mathcal G}.$
\end{thm}

It is well known that the integral generator of a $C$-distribution semigroup in a Banach space can have the empty $C$-resolvent set. On the other hand, the existence and polynomial boundedness of $C$-resolvent of $A$ on a certain exponential region ensures that the operator $C^{-1}AC$
generates a (C-DS). More precisely, we have the following.

\begin{thm}\label{herbs}
Let $a>0$, $b>0$, $\alpha>0$ and $e(a,b)\subseteq \rho_C(A).$ Suppose that the mapping $\lambda\mapsto(\lambda-A)^{-1}Cx$,
$\lambda\in e(a,b)$ is continuous for every fixed element $x\in E$, as well as that the operator family $\{(1+|\lambda|)^{-\alpha}(\lambda- A)^{-1}C : \lambda \in e(a,b)\}\subseteq L(E)$ is equicontinuous.
Set
\begin{align}\label{franci}
\mathcal{G}(\varphi)x:=(-i)\int_{\Gamma}\hat{\varphi}(\lambda)(\lambda-A)^{-1}Cx\,d\lambda,
\;\;x\in E,\;\varphi\in\mathcal{D},
\end{align}
with $\Gamma$ being the upwards oriented boundary of region $e(a,b)$.
Then $\mathcal{G}$ is a (C-DS) generated by $C^{-1}AC$.
\end{thm}

\begin{proof}
Without loss of generality, we may assume that, for every $x\in E,$ the mapping $\lambda\mapsto (\lambda-A)^{-1}Cx$ is analytic
on some open neighborhood of the region $e(a,b);$ cf. \cite[Proposition 2.16]{sic}. By the argumentation given in the proof of \cite[Theorem 3.1.27]{knjigah}, it readily follows that $\mathcal{G}\in {\mathcal D}^{\prime}_{0}(L(E)),$ as well as that
$\mathcal{G}(\varphi)C^{l}((z-A)^{-1}C)^{m}=C^{l}((z-A)^{-1}C)^{m}\mathcal{G}(\varphi),$ $\varphi \in {\mathcal D}$ ($m,\ l\in {\mathbb N}_{0}$).
Suppose that $B$ is a bounded subset of ${\mathcal D}.$ Then there exists $\tau>0$ such that $B$ is contained and bounded in ${\mathcal D}_{[-\tau,\tau]}.$ Since
\begin{equation}\label{mos-def}
\hat{\varphi}(\lambda)=\frac{(-1)^{n}}{2\pi \lambda^{n}}\int^{\infty}_{-\infty}e^{\lambda t}\varphi^{(n)}(t)\, dt,\quad \varphi \in {\mathcal D},\ \lambda \in {\mathbb C} \setminus \{0\} ,
\end{equation}
we obtain that for each $n\in {\mathbb N}$ there exists $c_{n}>0$ such that for each $\varphi \in B$ we have $|\hat{\varphi}(\lambda)|\leq c_{n}e^{\tau \Re \lambda}|\lambda|^{-n},$ $\lambda \in {\mathbb C} \setminus \{0\} .$ Keeping in mind this estimate and the equicontinuity of the family $\{(1+|\lambda|)^{-\alpha}(\lambda- A)^{-1}C : \lambda \in e(a,b)\}$, it can be simply proved that ${\mathcal G}$ is boundedly equicontinuous. Similarly as in the proof of \cite[Theorem 3.1.27]{knjigah}, we have that $\mathcal{G}(\varphi)A\subseteq A{\mathcal G}(\varphi),$ $\varphi \in {\mathcal D}$ and (\ref{dkenk}) holds. By Theorem \ref{fundamentalna}(i), we get that ${\mathcal G}$ satisfies (C.S.1).
 In order to prove (C.S.2), suppose that ${\mathcal G}(\varphi)x'=0$, $\varphi \in {\mathcal D}_{0}$ for some element $x'\in E.$ Owing to \cite[Theorem 4.4(ii)]{sic}, we know that there exist numbers $n\in {\mathbb N},$  $n>\alpha+1,$ and $\tau'>0$ such that the operator $A$ ($C^{-1}AC$) is a subgenerator (the integral generator) of a locally equicontinuous non-degenerate $n$-times integrated $C$-semigroup $(S_{n}(t))_{t\in [0,\tau')},$ given by
$$
S_{n}(t)x=\frac{1}{2\pi i}\int_{\Gamma}e^{\lambda t}\lambda^{-n}\bigl(\lambda-A\bigr)^{-1}Cx\, d\lambda,\quad x\in E,\ t\in [0,\tau').
$$
Integration by parts implies that $\int^{\infty}_{0}\varphi^{(n)}(t)e^{\lambda t}\, dt=(-1)^{n}\lambda^{n}\int^{\infty}_{0}\varphi (t)e^{\lambda t}\, dt,$ $\lambda \in {\mathbb C},$ $\varphi \in {\mathcal D}_{(0,\tau')}.$ Exploiting this equality, the Fubini theorem, and the foregoing arguments, it can be simply verified that:
$$
{\mathcal G}(\varphi)x=(-1)^{n}\int^{\infty}_{0}\varphi^{(n)}(t)S_{n}
(t)x\, dt,\quad x\in E,\ \varphi \in {\mathcal D}_{(0,\tau')}.
$$
This, in particular, holds with $x=x',$ so that Lemma \ref{polinomi} implies that there exist elements $x_{0},\cdot \cdot \cdot,x_{n-1}\in E$ such that $S_{n}(t)x'=\sum^{n-1}_{j=0}t^{j}x_{j},$ $t\in [0,\tau').$ Plugging $t=0,$ we get that $x_{0}=0.$ Hence,
$$
A\sum \limits^{n-1}_{j=1}\frac{t^{j+1}}{j+1}x_{j}=\sum_{j=1}^{n-1}t^{j}x_{j}-\frac{t^{n}}{n!}Cx',\quad t\in [0,\tau'),
$$
which implies $x_{1}=\cdot \cdot \cdot =x_{n-1}=0,$ and consequently, $x'=0,$ because $(S_{n}(t))_{t\in [0,\tau')}$ is non-degenerate.
We have proved that ${\mathcal G}$ is a (C-DS). By Theorem \ref{fundamentalna}(ii), the integral generator of ${\mathcal G}$ is the operator $C^{-1}AC.$
\end{proof}

\begin{rem}\label{filipa-1}
\begin{itemize}
\item[(i)]
In the proof of \cite[Theorem 3.1.27]{knjigah}, the equation (\ref{dkenk}) and the structural theorem for vector-valued distributions supported by a point have been essentially used in proving the property (C.S.2) for ${\mathcal G}.$ Observe that we do  not assume here that the space $E$ is admissible.
\item[(ii)] Now we would like to explain how one can prove the property (C.S.1) for ${\mathcal G}$ by using direct computations.
Let us fix two test functions $\varphi,$ $\psi \in {\mathcal D},$
an element $x\in E$ and a number $z\in \rho_{C}(A) \setminus e(a,b).$ Using the generalized resolvent equation \cite[(6)]{knjigaho}, we get that the operator family $\{(1+|\lambda|)^{-1}(\lambda- A)^{-1}C((z-A)^{-1}C)^{\lceil \alpha \rceil +1} : \lambda \in e(a,b)\}\subseteq L(E)$ is equicontinuous. Set $y:=((z-A)^{-1}C)^{\lceil \alpha \rceil +1}x.$ Making use of the equation (\ref{mos-def}) with $n=1,$ the partial integration, and the Cauchy formula, one obtains:
\begin{align}
\notag {\mathcal G}(\varphi)y&=\frac{(-1)}{2\pi i}\int_{\Gamma}\! \int_{-\infty}^{\infty}\frac{1}{\lambda}e^{\lambda t}\varphi^{\prime}(t)\bigl (\lambda-A\bigr)^{-1}Cy\, dt \, d\lambda
\\\notag & =
\frac{(-1)}{2\pi i}\int_{\Gamma}\! \int_{0}^{\infty}\frac{1}{\lambda}e^{\lambda t}\varphi^{\prime}(t)\bigl (\lambda-A\bigr)^{-1}Cy\, dt\, d\lambda
\\\notag & =\frac{1}{2\pi i}\int_{\Gamma}\Biggl[ \frac{\varphi (0)}{\lambda} + \int^{\infty}_{0}e^{\lambda t}\varphi(t)\, dt \Biggr]\bigl (\lambda-A\bigr)^{-1}Cy\, d\lambda
\\\label{mos-def-1} & =\frac{1}{2\pi i}\int_{\Gamma}\! \int^{\infty}_{0}e^{\lambda t}\varphi(t)\bigl (\lambda-A\bigr)^{-1}Cy\, dt \, d\lambda,\quad \varphi \in {\mathcal D}.
\end{align}
This simply implies
\begin{align}\label{mos-def-2}
{\mathcal G}(\varphi \ast_{0}\psi)y=\frac{1}{2\pi i}\int_{\Gamma} \  \int_{0}^{\infty} \int_{0}^{\infty} e^{\lambda (t+s)}\varphi(t)\psi (s)\bigl (\lambda-A\bigr)^{-1}Cy\, dt\, ds\, d\lambda.
\end{align}
Let $\theta \in (0,\pi/2)$ be a fixed angle. Considering, for a sufficiently large number $R>0,$ the positively oriented curve $\Gamma':=\Gamma_{1}' \cup \Gamma_{2}' \cup \Gamma_{3}' \cup \Gamma_{4}',$ where
$\Gamma_{1}':=\{t-iR\cos \theta : t\in [-R\sin \theta , a^{-1}\ln (R\cos \theta)]\},$  $\Gamma_{2}':=\{\lambda \in \Gamma: |\Im \lambda|\leq R\cos \theta\},$ $\Gamma_{3}':=\{t+iR\cos \theta : t\in [-R\sin \theta , a^{-1}\ln (R\cos \theta)]\}$ and
$\Gamma_{4}':=\{Re^{i\vartheta} : \vartheta \in [\theta +(\pi/2),(3\pi/2)-\theta]\},$ and applying the Cauchy formula, we get that
\begin{equation}\label{mos-def-3}
\frac{1}{2\pi i}\int_{\Gamma}\frac{\int^{\infty}_{0}e^{\lambda t}\varphi (t)\, dt}{\lambda-\eta}\, d\lambda=0,\quad \varphi \in {\mathcal D},\ \eta \notin \Gamma.
\end{equation}
Let $\Gamma_{1}$ be the positively oriented boundary of the region $e(a_{1},b_{1}),$ where $0<a_{1}<a$ and $b_{1}>b.$ Noticing that, for every fixed number $\lambda \in \Gamma,$ the operator $A=\lambda$ generates the strongly continuous semigroup $(e^{\lambda t})_{t\geq 0}$ on ${\mathbb C},$ and that every $C$-distribution semigroup on a Banach space is uniquely determined by its generator, we can apply \cite[Theorem 3.1.27]{knjigah} in order to see that
\begin{equation}\label{integral-lkj}
(-i)\int_{\eta \in \Gamma_{1}}\frac{\hat{\psi}(\eta)}{\eta-\lambda}\, d\eta=\int^{\infty}_{0}e^{\lambda t}\psi(t)\, dt.
\end{equation}
Set $\hat{\varphi}_{+}(\lambda):=\int^{\infty}_{0}e^{\lambda t}\varphi (t)\, dt,$ $\lambda \in {\mathbb C}.$
Using the Cauchy formula, (\ref{mos-def-1}), (\ref{mos-def-3})-(\ref{integral-lkj}), the resolvent equation, and the Fubini theorem, we obtain:
\begin{align*}
{\mathcal G}(\varphi ){\mathcal G}(\psi)y&=\frac{(-1)}{2\pi}\int_{\Gamma} \! \int_{\Gamma_{1}}\hat{\varphi}_{+}(\lambda)\hat{\psi}(\eta)\frac{\bigl(\lambda-A\bigr)^{-1}C^{2}y-\bigl(\eta-A\bigr)^{-1}C^{2}y}{\eta -\lambda}\, d\eta \, d\lambda
\\& =\frac{1}{2\pi i}\int_{\Gamma}\! \int^{\infty}_{0}\hat{\varphi}_{+}(\lambda)e^{\lambda t}\psi(t) \bigl(\lambda-A\bigr)^{-1}C^{2}y\, dt\, d\lambda
\\ &-\frac{1}{2\pi}\int_{\Gamma_{1}} \!  \Biggl( \int_{\Gamma}\frac{\int^{\infty}_{0}e^{\lambda t}\varphi (t)\, dt}{\lambda-\eta}\, d\lambda\Biggr) \hat{\psi}(\eta)\bigl(\eta-A\bigr)^{-1}C^{2}y\, d\eta
\\&=\frac{C}{2\pi i}\int_{\Gamma}\! \int^{\infty}_{0}\hat{\varphi}_{+}(\lambda)e^{\lambda t}\psi(t) \bigl(\lambda-A\bigr)^{-1}Cy\, dt\, d\lambda
\\& =\frac{C}{2\pi i}\int_{\Gamma} \  \int_{0}^{\infty} \int_{0}^{\infty} e^{\lambda (t+s)}\varphi(s)\psi (t)\bigl (\lambda-A\bigr)^{-1}Cy\, ds\, dt\, d\lambda
\\&=C{\mathcal G}\bigl(\varphi \ast_{0} \psi \bigr)y.
\end{align*}
Because ${\mathcal G}(\zeta)$ commutes with the operator $((z-A)^{-1}C)^{\lceil \alpha \rceil +1}$ ($\zeta \in {\mathcal D}$), the above computation implies by (\ref{mos-def-2}) that ${\mathcal G}(\varphi) {\mathcal G}(\psi)x={\mathcal G}(\varphi \ast_{0} \psi)Cx$ and that (C.S.1) holds.
\item[(ii)] The assertion of \cite[Proposition 3.1.28(i)]{knjigah} continues to hold in locally convex spaces.
\end{itemize}
\end{rem}

In the remaining part of this section, we shall reconsider the definition of a regular distribution semigroup given by J. L. Lions \cite{li121} and prove some results on dense (C-DS)'s ((C-UDS)'s of $\ast$-class).
Suppose that $\mathcal{G}\in\mathcal{D}'_0(L(E))$ ($\mathcal{G}\in\mathcal{D}'^{\ast}_0(L(E))$) is boundedly equicontinuous. We analyze the following conditions for ${\mathcal G}$:

$(d_1)$:\quad $\mathcal{G}(\varphi*\psi)C=\mathcal{G}(\varphi)\mathcal{G}(\psi)$, $\varphi,\,\psi\in\mathcal{D}_0$ ($\varphi,\,\psi\in\mathcal{D}^{\ast}_0$),

$(d_2)$:\quad the same as (C.S.2),

$(d_3)$:\quad $\mathcal{R}(\mathcal{G})$ is dense in $E$,

$(d_4)$:\quad for every $x\in\mathcal{R}(\mathcal{G})$, there exists a function $u_x\in C([0,\infty):E)$ so that
$u_x(0)=Cx$ and $\mathcal{G}(\varphi)x=\int_0^{\infty}\varphi(t)u_x(t)\,dt$, $\varphi\in\mathcal{D}$ ($\varphi\in\mathcal{D}^{\ast}$),

$(d_5)$:\quad if $(d_2)$ holds then $(d_5)$ means $G(\varphi_+)C=\mathcal{G}(\varphi)$, $\varphi\in\mathcal{D}$ ($\varphi\in\mathcal{D}^{\ast}$).

Using the same arguments as in the Banach space case (\cite{knjigah}), we can prove the following theorem.

\begin{thm}\label{wu-tang}
\begin{itemize}
\item[(i)] Suppose that $\mathcal{G}\in\mathcal{D}'_0(L(E))$ ($\mathcal{G}\in\mathcal{D}'^{\ast}_0(L(E))$) is boundedly equicontinuous and $\mathcal{G}C=C\mathcal{G}$.
Then $\mathcal{G}$ is a $($C-DS$)$ ($($C-UDS$)$ of $\ast$-class) iff $(d_1)$, $(d_2)$ and $(d_5)$ hold.
\item[(ii)] Suppose that $\mathcal{G}\in\mathcal{D}'_0(L(E))$ ($\mathcal{G}\in\mathcal{D}'^{\ast}_0(L(E))$) satisfies $(d_1)$, $(d_2)$, $(d_3)$, $(d_4)$
and $\mathcal{G}C=C\mathcal{G}$. If  ${\mathcal G}$ is boundedly equicontinuous, then
$\mathcal{G}$ is a $($C-DS$)$ ($($C-UDS$)$ of $\ast$-class).
\item[(iii)] Let $\mathcal{G}$ be a $($C-DS$)$ ($($C-UDS$)$ of $\ast$-class). Then $\mathcal{G}$ satisfies $(d_4)$.
\end{itemize}
\end{thm}

\section{Stationary dense operators in locally convex spaces}
Following P. C. Kunstmann \cite{ku101}, we introduce the notion of a stationary dense operator in a sequentially complete locally convex space as follows.

\begin{defn}\label{pc-kunst}
A closed linear operator $A$ is said to be stationary dense iff
$$
n(A):=\inf\bigl\{k\in\mathbb{N}_0:D(A^m)\subseteq\overline{D(A^{m+1})}\text{ for all }m\geq k\bigr\}<\infty.
$$
\end{defn}

The abstract Cauchy problem
\[(ACP_1):\left\{
\begin{array}{l}
u\in C([0,\tau):[D(A)])\cap C^1([0,\tau):E),\\
u'(t)=Au(t),\;t\in[0,\tau),\\
u(0)=x,
\end{array}
\right.
\]
where $A$ is a closed linear operator on $E$ and $0<\tau \leq \infty,$
has been analyzed in a great number of research papers and monographs (see e.g. \cite{a22}-\cite{b42},
\cite{cha}-\cite{c62}, \cite{d81}-\cite{engel}, \cite{hiber1}-\cite{ki90},
\cite{komatsulocally}-\cite{knjigaho}, \cite{ko98},
\cite{ptica}-\cite{tica}, \cite{ku101}-\cite{li121},
\cite{me152}, \cite{pa11} and \cite{ush}-\cite{yosi}).
By a mild solution of problem $(ACP_1)$ we mean any continuous function $t\mapsto u(t;x),$ $t\in [0,\tau)$ such that $A\int^{t}_{0}u(s;x)\, ds=u(t;x)-x,$ $t\in [0,\tau).$

\begin{prop}\label{pc-kunst-isto}
Let $0<\tau \leq \infty $ and $n\in {\mathbb N}_{0}.$
\begin{itemize}
\item[(i)]
Suppose that the abstract Cauchy problem $(ACP_1)$ has a unique mild solution $u(t;x)$ for all $x\in D(A^{n}).$
Then $A$ is stationary dense and $n(A)\leq n.$
\item[(ii)] Suppose that $(S_{n}(t))_{t\in [0,\tau)}$ is a locally equicontinuous
$n$-times integrated semigroup generated by $A.$ Then $A$ is stationary dense and $n(A)\leq n.$
\end{itemize}
\end{prop}
\begin{proof}
One has to use the arguments given in that of \cite[Lemma 1.7]{ku101} (cf. also \cite[Remark 1.2(i)]{ku101}) and the fact that for any locally equicontinuous
$n$-times integrated semigroup $(S_{n}(t))_{t\in [0,\tau)}$ generated by $A$ the abstract Cauchy problem
$(ACP_1)$ has a unique mild solution for all $x\in D(A^{n}),$ given by
$
u(t;x)=S_{n}(t)A^{n}x+\sum^{n-1}_{j=0}\frac{t^{j}}{j!}A^{j}x,$ $t\in [0,\tau)$ (\cite{a22}).
\end{proof}
\begin{lem}\label{lema0} Let $A$ be a closed operator in sequentially complete locally convex space and let $(\lambda_n)\in\rho(A)$ be a sequence such that $\lim_{n\rightarrow\infty}|\lambda_n|=\infty$ and there exist $C>0$ and $k\geq-1$ such that for every $p\in\circledast$, there exists $q\in\circledast$, such that $p(R(\lambda_n:A)x)\leq C|\lambda_n|^kq(x)$ for all $x\in E$ and $n\in\mathbb N$. Then $A$ is stationary dense with $n(A)\leq k+2$.\end{lem}
\begin{proof}
Let $x\in D(A^{k+1})$. Then for every $p\in\circledast$, there exists $q\in\circledast$, such that $p(\lambda_n R(\lambda_n:A)x)\leq \|R(\lambda_n:A)\| q(Ax)$. Now
for $x\in D(A^{k+2})$, it follows $\lambda_n R(\lambda_n:A)x\in D(A^{k+3})$ for all $\mathbb N$. Furthermore $$p(\lambda_n R(\lambda_n:A)x-x)=p(R(\lambda_n:A)Ax)\leq \frac{C'}{|\lambda_n|}q(Ax).$$ Hence, $$x=\lim\limits_{n\longrightarrow\infty}\lambda_n R(\lambda_n:A)x$$ and $x$ belongs to the closure of $D(A^{k+3})$, which means that $A$ is stationary dense and $n(A)\leq k+2$.
\end{proof}

We say that the operator $A$ satisfy the condition $(EQ)$ if:\\
$A_{\infty}$ is a generator of an equicontinuous semigroup $T_{\infty}(t)$ in $D_{\infty}(A)$, i.e. for every $p\in\circledast$ there exists $q\in\circledast$ and $C$ such that
$$p(T_{\infty}(t)x)\leq C q(x),$$ for every $x\in D_{\infty}(A)$.\\

Using the results given in \cite{ku101}, ~\cite[Theorem 4.1]{ush} we can state similar results in our setting (E is sequentially complete locally convex space). Assume that $A$ is stationary dense, satisfies $(EQ)$, $n=n(A)$ and $F$ is the closure of $D(A^n)$ in $E$.
\begin{lem}\label{lema1234}
\begin{itemize}
\item[a)]$A_F$ is densely defined in $F$, where $A_F$ means the restriction of the operator $A$ on $F$;
\item[b)] \label{lema2} $\rho(A;L(E))=\rho(A_F;L(F))$ for all $\lambda\in\rho(A)$. Additionally for all $x\in E$ and $p\in\circledast$, there exist $n\in{\mathbb N}$, $C>0$ such that
$$p(R(\lambda:A)x)\leq C(1+|\lambda|)^n(p(R(\lambda:A_F)x)+1);$$
\item[c)]\label{lema3} The Fr\'echet spaces $D_{\infty}(A)$ and $D_{\infty}(A_F)$ coincide topologically and $A_{\infty}=(A_F)_{\infty}$;
\item[d)]$n(A)=\inf\{k\in{\mathbf N_0}\, :\, \overline{D(A^k)}\subset D_{\infty}(A)\}.$\end{itemize}
\end{lem}
\begin{proof}\begin{itemize}
\item[a)] It is obvious since $D(A^n)$ is dense in $F$ and $D(A^n)\subset D(A_F)$.
\item[b)]
Since $D(A^n)\subset D(A)$, $F$ is invariant under resolvent of $A$, we obtain $\rho(A;L(E))\subset\rho(A_F;L(F))$. Now, let $\lambda\in\rho(A_F;L(F))$. Then $\lambda-A$ is injective, so $R(\lambda:A)$ is an extension of $R(\lambda:A_F)$. Let $\mu\in\rho(A)$. By $R(\lambda:A)=(\mu-A)^nR(\lambda:A)R(\mu:A)^n$, we obtain that $\lambda\in\rho(A:L(E))$.\\

It holds, for $x\in E$ and all $p\in\circledast$, $$p(R(\lambda:A)x)\leq p_n(R(\mu:A)^nx)p_n(R(\lambda:A_F)x)p(R(\mu:A)^nx)$$
where $p_n(x)=\sup\limits_{p\in\circledast}\sum\limits_{i=1}^{n}p(A^ix)$. Note that $(D(A^n),p_n)$ is a Banach space. For $x\in D(A^n)$ we have
$$p_n(R(\lambda:A_F)x)=\sup\limits_{p\in\circledast}\sum\limits_{i=0}^{n}p(A^iR(\lambda:A_F)x)=$$
$$=\sup\limits_{p\in\circledast}\sum\limits_{i=0}^{n}p(({\lambda}^k+A^k-{\lambda}^k)R(\lambda:A_F)x)\leq$$
$$\leq\sup\limits_{p\in\circledast}\sum\limits_{i=0}^{n}{\Biggl(}{|\lambda|}^k
p(R({\lambda}:A_F)x)+
\sum\limits_{j=0}^{i-1}{|\lambda|}^j p(A^{i-1-j}x){\Biggr)}.$$
The last one inequality, together with the previous one gives the statement of the lemma.

\item[c)]It holds $D(A^{n+k})\hookrightarrow D((A_F)^k)\hookrightarrow D(A^k)$, which one holds for all $k\in{\mathbb N}_0$. Then $$\bigcap\limits_{k=1}^{\infty}D(A^{n+k})\hookrightarrow\bigcap\limits_{k=1}^{\infty}D(A_F^k)
\hookrightarrow\bigcap\limits_{k=1}^{\infty}D(A^k),$$ which gives that $D_{\infty}(A)$ and $D_{\infty}(A_F)$ coincide topologically and $A_{\infty}=(A_F)_{\infty}$.
\item[d)] From the previous lemma, $D_{\infty}(A)=D_{\infty}(A_F)$ and $D(A^{n+1})\subset D(A_F)$, we obtain the conclusion of the lemma.\end{itemize}\end{proof}

\begin{thm}\label{eqeq} Let $A$ be a stationary dense operator in $E$ with non-empty resolvent. Then $\rho(A;L(E))=\rho(A_{\infty};L(D_{\infty}(A)))$.\end{thm}
\begin{proof}
By Lemma \ref{lema1234} a) we have that $A_F$ is dense in $F$ and by b) from the same lemma it has non-empty resolvent. Then using the proof of ~\cite[Theorem 2.3]{ku101} and Lemma \ref{lema1234} c), we obtain that $\rho(A;L(E))=\rho(A_{\infty};L(D_{\infty}(A)))$.
\end{proof}
Next example show us if $A$ is not stationary dense operator in $E$, then the conclusion of Theorem \ref{eqeq} does not hold. \begin{example}
Let we define the space $S_j$ as $$S_j=\{\varphi\in{\mathcal C}^{\infty}({\mathbb R})\, :\, p_j(x)=\sup\limits_{\alpha+\beta\leq j}\|x^{\beta}D^{\alpha}\varphi(x)\|_{L^2({\mathbb R})}<\infty\}.$$ Then the test space for tempered distributions $S({\mathbb R})$ can be defined as ${\mathcal S}({\mathbb R})=\lim\limits_{j\rightarrow\infty}\mbox{proj} S_j$. Let $E={\mathcal S}({\mathbb R})$, ($E$ is a Fr\'echet space as a projective limit of Banach spaces, so $E$ is a sequentially complete locally convex space). Define $A=-\frac{d}{dt}$ on $E$ with domain $D(A)=\{f\in E\, :\, f(0)=0\}$. The operator $A$ is not stationary dense on $E$. Note that $D_{\infty}(A)=\{0\}$ and $\rho(A_{\infty})=\{\lambda\in{\mathbb C}\, :\, \lambda\neq 0\}$. For $f\in E$, $\lambda\in{\mathbb C}$, and $\Re\lambda>s$, we have $$(\lambda-A)^{-1}f=\int\limits_0^{\infty}e^{-(\lambda-s)t}f(t)\, dt$$ belongs in $E$. Then $\rho(A)=\{\lambda\in{\mathbb C}\, :\, \Re\lambda>s\}$. Therefore, we obtain that the conclusion of Theorem \ref{eqeq} does not hold. The same conclusion can be made in the ultradistribution case.
Consider the spaces $$E_{(h)}=\{f\in C^{\infty}({\mathbb R})\, :\, f(0)=0,\, p_n(f)=\sup\limits_{k\in{\mathbb N}_0}\sup\limits_{t\geq k}\frac{h^n|{f}^{(n)}(t)|}{M_n}<+\infty\}$$
$$E_{\{h\}}=\{f\in C^{\infty}({\mathbb R})\, :\, f(0)=0,\, p_n(f)=\sup\limits_{k\in{\mathbb N}_0}\sup\limits_{t\geq k}\frac{h^n|{f}^{(n)}(t)|}{M_n\prod\limits_{i=1}^nr_i}<+\infty\}$$ for $(r_i)$ monotonically increasing positive sequence and $(M_n)$ satisfying $(M.1)$ and $(M.3)'$. The spaces $E_{(h)}$ and $E_{\{h\}}$ are Fr\'echet spaces. Let $A=-\frac{d}{ds}$ with domain $D(A)=\{f\in{E_h}\, : \, f(0)=0, Af\in E_h\}$, where $E_h$ stands for both spaces. Let $p_n\in\circledast$ be a seminorm in $E_{\{h\}}$.
The previous consideration in distribution case for $E={\mathcal S}({\mathbb R})$ is similar and more simple then the  case with spaces $E_{(h)}$ and $E_{\{h\}}$.
\end{example}
\begin{thm}\label{kor1} Let $A$ be a stationary dense operator in a sequentially complete locally convex space $E$, $n=n(A)$ and $F=\overline{D(A^n)}$. Then $A$ generates a distribution semigroup in $E$ if and only if $A_F$ generates a distribution semigroup in $F$.\end{thm}
\begin{proof}The proof of this theorem is direct consequence of \cite[Theorem 2.7]{tica}, Lemma \ref{lema0} and Lemma \ref{lema1234} b). \end{proof}
Following two results are due to T. Ushijima and we can restate it in locally convex case as following.
\begin{thm}\label{pom1} Let $E$ be a sequentially complete locally convex space and $A$ be a closed operator with dense $D(A^{\infty})$. Then $A$ has the property $(EQ)$ if and only if there exists logarithmic region $\Omega_{\alpha,\beta}$ and $k\in{\mathbb N}$ and $C>0$ such that $$ ||R(\lambda : A)||\leq C(1+|\lambda|)^k, \quad \lambda\in \Omega_{\alpha,\beta}.$$
\end{thm}
\begin{thm} Let $E$ be a sequentially complete locally convex space and $A$ linear operator on $E$. The following conditions are equivalent:
\begin{itemize}
\item[i)] $A$ is the generator of a distribution semigroup $G$;
\item[ii)] $A$ is well-posed and densely defined;
\item[iii)] $A$ satisfy the condition $(EQ)$;
\item[iv)] $A$ is densely defined, and there exist an adjoint logarithmic region $\Omega_{\alpha,\beta}$ and $k\in{\mathbb N}$ and $C>0$ such that
$$ ||R(\lambda : A)||\leq C(1+|\lambda|)^k, \quad \lambda\in \Omega_{\alpha,\beta}.$$
\end{itemize}
\end{thm}
\begin{proof}
We will prove $i)\Rightarrow ii)\Rightarrow iii)\Rightarrow iv)\Rightarrow i)$. The statements $i)\Leftrightarrow ii)$ and $iv)\Rightarrow i)$ follow from \cite[Theorem 2.7]{tica} and $iii)\Rightarrow iv)$ follows by Theorem \ref{pom1}  It remains to show that $i)\Rightarrow iii)$ and $iv)\Rightarrow ii)$.\\
$i)\Rightarrow iii)$: By the definition of the distribution semigroup, we can conclude that $D(A^{\infty})$ contains
${\mathcal R}(G)$. Since ${\mathcal R}(G)$ is dense in $E$, $D(A^{\infty})$ is dense in $E$. By the results in the third section from \cite{ush1}, we obtain that $$({\boldsymbol\lambda}-{\mathbf A}_{\infty})^{-1}(1\otimes x)(\hat{\varphi})=G(\varphi)x,$$ for any $\varphi\in\DD$ and $x\in D_{\infty}(A)$. Note that with $({\boldsymbol\lambda}-{\mathbf A}_{\infty})^{-1}$ is denoted the generalized resolvent. Let $f_{\lambda}(t)=\mu(t)e^{-\lambda t}$, where $\mu\in\DD$ and $\mu(t)=1$ for $|t|\leq1$. Define the operator $R_{\infty}(\lambda)=G(f_{\lambda})$. Since $G\in\DD'_0(L(E))$, for all $p\in\circledast$, there exists $q\in\circledast$, $k\in{\mathbb N}$, $C>0$ such that for all $x\in E$, $\lambda\in{\mathbb C}$, $\Re\lambda\geq0$ $$p(R_{\infty}(\lambda)x)\leq C (1+|\lambda|)^kq(x).$$
Now, for any $x\in D_{\infty}(A)$, $\varphi\in\DD$, $a>0$,

$$\frac{1}{i}\int\limits_{a-i\infty}^{a+i\infty}\hat{\varphi}(\lambda)R_{\infty}(\lambda)x\, d\lambda=\int\limits_{a-i\infty}^{a+i\infty}{\Bigl(}\frac{1}{2\pi i}\int\limits_{-\infty}^{\infty}\varphi(t) e^{t\lambda}\, dt{\Bigr)}G(f_{\lambda})\, d\lambda\, x=$$

$$=G{\Bigl(}\mu(s)\frac{1}{2\pi i}\int\limits_{a-i\infty}^{a+i\infty}\, d\lambda\int\limits_{-\infty}^{+\infty}\varphi(t)e^{(t-s)\lambda}\, dt{\Bigr)}x=G(\mu\cdot\varphi)x=$$
$$=G(\varphi)x=({\boldsymbol{\lambda}}-{\textbf A}_{\infty})^{-1}(1\otimes x)(\hat{\varphi}).$$

Again, using third section in \cite{ush1}, we obtain $$\int_0^{\infty}\varphi(t)T_{\infty}(t)x\, dt=G(\varphi)x,$$ for all $\varphi\DD$ and $x\in D_{\infty}(A)$, which gives $iii)$.
$v)\Rightarrow ii)$: It is same like in $(v)\Rightarrow (ii)$ of ~\cite[Theorem 4.1]{ush}.
\end{proof}

\begin{thm} Let $A$ be a closed operator in $E$. Then $A$ generates a distribution semigroup in $E$ if and only if $A$ is stationary dense and $A_{\infty}$ generates an equicontinuous semigroup in $D_{\infty}(A)$.\end{thm}
\begin{proof}
Let $A$ generates distribution semigroup in $E$. By simpler version of \cite[Theorem 2.7]{tica} and Lemma \ref{lema0} $A$ is stationary dense.
Now, let $n=n(A)$ and $F=\overline{D(A^n)}$. Then by Lemma \ref{lema1234} a) and Theorem \ref{kor1} $A_F$ generates a dense distribution semigroup in $F$. Then $(A_F)_{\infty}$ generates an equicontinuous semigroup in $D_{\infty}(A_F)$. By Lemma \ref{lema1234} c) follows the conclusion.\\
Opposite direction. Let $A$ is stationary dense and $A_{\infty}$ generates an equicontinous semigroup and $n=n(A)$ and $F=\overline{D(A^n)}$. Then $A_F$ generates a distribution semigroup in $F$. By Theorem \ref{kor1} we obtain that $A$ generates a distribution semigroup in $E$.
\end{proof}
\begin{thm} Let $A$ be a closed operator in $E$. Then $A$ generates an exponential distribution semigroup in $E$ if and only if $A$ is stationary dense and $A_{\infty}$ generates a quasi-equicontinuous semigroup in $D_{\infty}(A)$.\end{thm}
The proof is direct consequence having on mind the results of exponential distribution semigroups listed before.

\end{document}